\documentclass[reqno,11pt,twoside]{article}
\usepackage[hmargin=1.25in, vmargin=1.25in, a4paper, centering]{geometry}
\usepackage{amsmath, amsthm,amssymb}
\usepackage{fourier}
\usepackage{ebgaramond}
\usepackage{graphicx}
\usepackage{xcolor}
\usepackage[ font=small, labelfont=bf, labelsep=colon]{caption}
\usepackage{etoolbox}
\usepackage{longtable,booktabs,array}
\usepackage{algorithm}
\usepackage{algpseudocode}
\usepackage{marvosym}
\usepackage[section]{placeins}
\usepackage{tocloft}

\numberwithin{equation}{section}
\theoremstyle{definition}
\newtheorem{defi}{Definition}[section]

\theoremstyle{plain}
\newtheorem{theorem}[defi]{Theorem}
 \newtheorem{prop}[defi]{Proposition}
\newtheorem{lemma}[defi]{Lemma}
\newtheorem{cor}[defi]{Corollary}

\theoremstyle{remark}
\newtheorem{remark}[defi]{Remark}
\newcommand{\addQEDstyle}[2]{\AtBeginEnvironment{#1}{\pushQED{\qed}\renewcommand{\qedsymbol}{#2}}\AtEndEnvironment{#1}{\popQED}}
\addQEDstyle{remark}{$\blacktriangleleft$}
\newcommand{\R}{\mathbb{R}}

\newcommand{\ZZ}{\mathbb{Z}}
\newcommand{\N}{\mathcal{N}}

\newcommand{\Ph}{\mathbb{P}}

\newcommand{\Dir}{{\mathrm{D}}}
\newcommand{\Neu}{{\mathrm{N}}}
\newcommand{\dr}{\mathrm{d}}

\renewcommand{\epsilon}{\varepsilon}
\newcommand{\WeylC}[1]{C_{#1}}
\newcommand{\WeylCB}[1]{C_{\mathrm{b},#1}}
\newcommand{\disk}{\mathbb{D}}
\newcommand{\Pt}{\mathcal{P}}
\newcommand{\Lfloor}{\left\lfloor}
\newcommand{\Rfloor}{\right\rfloor}
\newcommand{\entire}[1]{\Lfloor #1 \Rfloor}
\newcommand{\ceiling}[1]{\left\lceil#1\right\rceil}
\newcommand{\Q}{\mathbb{Q}}
\DeclareMathOperator{\cosU}{\overline{cos}}
\DeclareMathOperator{\cosL}{\underline{cos}}
\DeclareMathOperator{\arccosU}{\overline{arccos}}
\DeclareMathOperator{\arccosL}{\underline{arccos}}

\newcommand{\ball}[1]{\mathbb{B}^{#1}}

\renewcommand\footnotemark{}
\usepackage[nobottomtitles,pagestyles]{titlesec}
\titleformat{\section}
{\normalfont\large\bfseries}
{\filcenter\S\thesection.}{1ex}{\filcenter}
\widenhead*{0.5in}{0.5in}
\renewpagestyle{plain}[\bfseries]{\footrule\setfoot{}{\thepage}{}}
\newpagestyle{mystyle}[\bfseries]{
\headrule
\sethead[\thepage][][P\'{o}lya's conjecture for Euclidean balls]{N. Filonov, M. Levitin, I. Polterovich, and D. A. Sher}{}{\thepage}
}
\usepackage[colorlinks,allcolors=blue,pagebackref]{hyperref}
\usepackage{pifont}
\renewcommand*{\backrefalt}[4]{%
\ifcase #1 %
No citations%
\or
\ding{43}~p.~#2%
\else
\ding{43}~pp.~#2%
\fi}
\newcommand{\mydoi}[1]{\href{https://doi.org/#1}{doi: #1}}
\newcommand{\myarXiv}[1]{\href{https://arxiv.org/abs/#1}{arXiv: #1}}
\begin{document}
\thispagestyle{plain}
\title{%
P\'{o}lya's conjecture for Euclidean balls%
\footnote{A \texttt{Mathematica} script used for a computer-assisted part of the paper is available for download at \url{https://michaellevitin.net/polya.html}}
\footnote{{\bf MSC(2020): }Primary 35P15. Secondary 35P20, 33C10, 11P21.}%
\footnote{{\bf Keywords: } Laplacian, eigenvalues, Weyl's law, lattice points, Bessel functions, Bessel phase functions, zeros of Bessel functions and their derivaives}%
}
\author{
Nikolay Filonov
\thanks{%
\textbf{N. F.: }St. Petersburg Department
of Steklov Institute of Mathematics of RAS,
Fontanka 27, 191023, St.Petersburg, Russia;
St. Petersburg State University,
University emb. 7/9,
199034, St.Petersburg, Russia; 
\href{mailto:filonov@pdmi.ras.ru}{\nolinkurl{filonov@pdmi.ras.ru}}%
}
\and
Michael Levitin\hspace{-3ex}
\thanks{%
\textbf{M. L.: }Department of Mathematics and Statistics, University of Reading, 
Pepper Lane, Whiteknights, Reading RG6 6AX, UK;
\href{mailto:M.Levitin@reading.ac.uk}{\nolinkurl{M.Levitin@reading.ac.uk}}; \url{https://www.michaellevitin.net}%
}
\and 
Iosif Polterovich
\thanks{%
\textbf{I. P.: }D\'e\-par\-te\-ment de math\'ematiques et de statistique, Univer\-sit\'e de Mont\-r\'eal, 
CP 6128 succ Centre-Ville, Mont\-r\'eal QC  H3C 3J7, Canada;
\href{mailto:iossif@dms.umontreal.ca}{\nolinkurl{iossif@dms.umontreal.ca}}; \url{https://www.dms.umontreal.ca/\~iossif}%
}
\and
David A. Sher
\thanks{%
\textbf{D. A. S.:  }Department of Mathematical Sciences, DePaul University, 2320 N. Kenmore Ave, 60614, Chicago, IL, USA;
\href{mailto:dsher@depaul.edu}{\nolinkurl{dsher@depaul.edu}}
}
}
\date{\small Revised version, 2 May 2023\\to appear in Invent. Math.} 
\maketitle

\begin{abstract}  The celebrated P\'{o}lya's conjecture (1954)  in spectral geometry  states that the eigenvalue counting functions of the Dirichlet and Neumann Laplacian on a bounded Euclidean domain  can be estimated from above and below, respectively, by the 
leading term of Weyl's asymptotics. P\'{o}lya's conjecture is known to be true for domains which tile Euclidean space, and, in addition, for some special domains in higher dimensions.  
In this paper, we prove  P\'{o}lya's conjecture for the disk, making it
the first non-tiling planar domain for which the conjecture is verified.  We also confirm  P\'{o}lya's conjecture
for arbitrary planar sectors, and, in the Dirichlet case, for balls of any dimension. Along the way, we develop
the known links between the spectral problems in the disk and certain lattice counting problems.
A key novel ingredient is the observation, made in recent work of the last named author, that
the corresponding eigenvalue and lattice counting functions are related not only asymptotically,
but in fact satisfy certain uniform bounds. Our proofs are purely  analytic, except for a rigorous computer-assisted argument needed to cover 
the short interval of values of the spectral parameter in the case of the Neumann problem in the disk.
\end{abstract}

{\small \tableofcontents}
\pagestyle{mystyle}
\section{Weyl's law and P\'{o}lya's conjecture}\label{sec:WL}
Let $\Omega\subset\mathbb{R}^d$ be a bounded domain. Consider the Dirichlet eigenvalue problem for the Laplacian $$-\Delta:=-\sum_{j=1}^d\frac{\partial^2}{\partial x_j^2}$$ in $\Omega$:
\begin{equation}\label{eq:dirL}
\begin{split}    
-\Delta u=\lambda u\qquad&\text{in }\Omega,\\ 
u=0\qquad&\text{on }\partial\Omega.
\end{split}
\end{equation}
It is well known that the spectrum of \eqref{eq:dirL} is discrete and consists of isolated eigenvalues of finite multiplicity accumulating to $+\infty$,
\[
0< \lambda_1(\Omega)\le \lambda_2(\Omega)\le\dots\le \lambda_n(\Omega)\le \dots,
\]
which we enumerate with account of multiplicities. 

Similarly, assuming additionally that $\partial\Omega$ is Lipschitz,  consider the Neumann eigenvalue problem 
\begin{equation}\label{eq:neuL}
\begin{split}    
-\Delta u=\mu u \qquad&\text{in }\Omega,\\ 
\partial_n u=0\qquad&\text{on }\partial\Omega,
\end{split}
\end{equation}
where $\partial_n u = \left.\langle \nabla u, n\rangle\right|_{\partial\Omega}$ denotes the normal derivative of $u$ with respect to the exterior unit normal $n$ on the boundary.    
The spectrum of \eqref{eq:neuL} again consists of isolated eigenvalues of finite multiplicity accumulating to $+\infty$,
\[
0=\mu_1(\Omega)\le \mu_2(\Omega)\le\dots\le \mu_n(\Omega)\le \dots,
\]
enumerated with account of multiplicities. 

Let, for $\lambda\in\R$, 
\[
\N^\Dir_\Omega(\lambda):=\#\left\{n: \lambda_n(\Omega)\le \lambda^2\right\}\qquad\text{and}\qquad\N^\Neu_\Omega(\lambda):=\#\left\{n: \mu_n(\Omega)\le \lambda^2\right\}
\]
denote the \emph{counting functions}\footnote{Strictly speaking, we are counting the number of eigenvalues less than or equal to a given $\lambda^2$, but such normalisation will be convenient to us throughout.} of the Dirichlet and Neumann eigenvalue problems on $\Omega$.\footnote{One can also define the counting functions using strict inequalities; this does not affect any of the results below.}
It follows from the variational principles for \eqref{eq:dirL} and \eqref{eq:neuL} that 
\[
\N^\Dir_\Omega(\lambda)\le \N^\Neu_\Omega(\lambda)
\]
for any $\lambda\ge0$.\footnote{This in fact can be improved to $\N^\Dir_\Omega(\lambda)+1\le \N^\Neu_\Omega(\lambda)$, see \cite{Fri} and \cite{Fil04}.} 
 
Under the assumptions stated above, the leading term asymptotics of the counting functions is given by  {\em Weyl's law} \cite{Wey11}, 
\begin{equation}\label{eq:Weylaw}
\N_\Omega(\lambda)= \WeylC{d} |\Omega|_d \lambda^d+R(\lambda),
\end{equation}
where $\N_\Omega(\lambda)$ denotes either $\N^\Dir_\Omega(\lambda)$ or $\N^\Neu_\Omega(\lambda)$,  $|\cdot|_d$ denotes the $d$-dimensional volume, $R(\lambda)=o\left(\lambda^d\right)$ as $\lambda\to+\infty$, and
\[
\WeylC{d}:=\frac{1}{(4\pi)^{\frac{d}{2}} \Gamma\left(\frac{d}{2}+1\right)}
\]
is the so-called \emph{Weyl constant}.  We refer to \cite{SV97} for a historical review, as well as numerous generalisations and improvements.

H. Weyl himself conjectured \cite{Wey12} a sharper version of  \eqref{eq:Weylaw} taking into account the boundary conditions: for $\Omega\subset\R^d$ with a piecewise smooth boundary, 
\begin{equation}\label{eq:twoterm}
\N_\Omega(\lambda)= \WeylC{d} |\Omega|_d \lambda^d \pm  \WeylCB{d} |\partial \Omega|_{d-1} \lambda^{d-1}+o\left(\lambda^{d-1}\right)\qquad\text{as }\lambda\to+\infty,
\end{equation}  
where the minus sign is taken for the Dirichlet boundary conditions and the plus sign for the Neumann ones, and  
\[
\WeylCB{d}:=\frac{1}{2^{d+1} \pi^\frac{d-1}{2}\Gamma\left(\frac{d+1}{2}\right)}.
\]
We note that for planar domains \eqref{eq:twoterm} takes the particularly simple form  
\begin{equation}\label{eq:twoterm2d}
\N_\Omega(\lambda)=\frac{\operatorname{Area}(\Omega)}{4\pi}\lambda^2 \pm \frac{\operatorname{Length}(\partial \Omega)}{4\pi}\lambda+o\left(\lambda\right).
\end{equation}

The two-term Weyl's law \eqref{eq:twoterm} remains open in full generality.  It has been proved by V. Ivrii \cite{Ivr80} under the condition that the set of 
periodic billiard trajectories in $\Omega$ has measure zero.  While this condition is conjectured to be satisfied for all Euclidean domains, it has been verified only for a few  classes,  such as convex analytic domains and polygons, see \cite{SV97} and references therein. Specifically for a disk, it was proved by N. Kuznetsov and B. Fedosov in \cite{kufe}.

Assuming that the two-term Weyl's asymptotics \eqref{eq:twoterm} holds for a domain $\Omega\subset\R^d$, we immediately obtain that for $\lambda$ above some \emph{sufficiently large but unspecified} value $\Lambda_1$ we have
\begin{equation}\label{eq:polya}
\N^\Dir_\Omega(\lambda)\le  \WeylC{d} |\Omega|_d \lambda^d  \le \N^\Neu_\Omega(\lambda).
\end{equation}
We refer also to \cite{Mel80} for results of the same kind in the Riemannian setting.

In 1954, G. P\'{o}lya \cite{Pol54} conjectured that the inequalities \eqref{eq:polya} hold for \emph{all} $\lambda\ge 0$.\footnote{In fact, P\'{o}lya's original conjecture was only for planar domains, and in a slightly different form.}
He later proved  this conjecture  in \cite{Pol61} for \emph{tiling domains} $\Omega$: that is, domains such that $\R^d$ can be covered, up to a set of measure zero, by a disjoint union of copies of $\Omega$. In fact, in the Neumann case, some additional assumptions were imposed in \cite{Pol61} that have been removed in \cite{Kel66}. It has been also shown that P\'{o}lya's conjecture in the Dirichlet case holds for a Cartesian product $\Omega=\Omega_1\times\Omega_2\subset \mathbb{R}^{d_1+d_2}$ if it holds for $\Omega_1\subset\mathbb{R}^{d_1}$ with $d_1\ge 2$, and $\Omega_2\subset\mathbb{R}^{d_2}$ is bounded,  see \cite[Theorem 2.8]{Lap}. 
For general domains, somewhat weakened versions of \eqref{eq:polya} are known to hold as a consequence of the so-called \emph{Berezin--Li--Yau} inequalities: we have
\[
\left(\frac{d}{d+2}\right)^{d/2}\N^\Dir_\Omega(\lambda)\le  \WeylC{d} |\Omega|_d \lambda^d \le \frac{d+2}{2} \N^\Neu_\Omega(\lambda) 
\]
for all $\lambda\ge 0$,  see \cite{LiYau}, \cite{Kro}, and \cite{Lap}. We refer also to \cite{Lin17}, \cite{KLS19}, \cite{FLP}, and \cite{Fre22} for some recent results on P\'olya's conjecture and further interesting links to other problems in spectral geometry.  

\begin{remark}\label{rem:polyaev}
P\'{o}lya's conjecture \eqref{eq:polya} can be equivalently restated as the inequalities for the eigenvalues (instead of the counting functions), 
\begin{equation}\label{eq:polyaev}
\mu_{n+1}(\Omega)\le \left(\WeylC{d} |\Omega|_d\right)^{-\frac{2}{d}} n^{\frac{2}{d}}\le\lambda_n(\Omega)
\end{equation} 
for all $n\ge 1$. It is known that inequalities \eqref{eq:polyaev} hold for any domain  in any dimension for $n=1, 2$.
 In particular, for $n=1$ this follows from the celebrated Faber--Krahn and Szeg\H{o}--Weinberger inequalities, and for $n=2$ in the Dirichlet case from the Krahn--Szego inequality, see \cite{Henrot}. For $n=2$ in the Neumann case, we refer to \cite{GNP09}, \cite{BuHe19}.
These are the only eigenvalues for which it is known in full generality.  We refer also to \cite{Fre} for further results on the validity of the Dirichlet P\'olya's conjecture for low eigenvalues in higher dimensions. 
\end{remark}

Remarkably, since balls do not tile the space, P\'{o}lya's conjecture has so far remained open for Euclidean balls, including planar disks.\footnote{As stated in \cite[p. 638] {Lap12}: ``Remarkably this conjecture still remains open even for such a simple domain as the disc, where the eigenvalues of the Dirichlet Laplacians could be calculated via the roots of Bessel functions.'' See also \cite[p. 1366]{FLW09} and \cite[p. 66]{lauarxiv}.} 
Although all the eigenvalues of the Dirichlet and Neumann Laplacians on the unit disk are explicitly known in terms of zeros of the Bessel functions or their derivatives, see \S\ref{sec:evdisk} below, in each case the spectrum is given by a \emph{two-parametric} family, and rearranging it into a single monotone sequence appears to be an unfeasible task.   

\medskip

Let $\ball{d}\subset\mathbb{R}^d$ be the $d$-dimensional unit ball. Then $|\ball{d}|_d=\frac{\pi^{d/2}}{\Gamma\left(\frac{d}{2}+1\right)}$. Therefore the leading Weyl's term  in \eqref{eq:Weylaw} for $\ball{d}$ becomes
\begin{equation}\label{eq:Wd}
W_d(\lambda):=\WeylC{d} |\ball{d}|_d \lambda^d= w_d\lambda^d,\qquad w_d=\frac{1}{2^d\left(\Gamma\left(\frac{d}{2}+1\right)\right)^2},
\end{equation}
in particular 
\[
W_2(\lambda)=\frac{\lambda^2}{4}\qquad\text{and}\qquad W_3(\lambda)=\frac{2\lambda^3}{9\pi}.
\]

The main results of this paper address the validity of P\'{o}lya's conjecture for disks and balls.  Namely, we prove the following results.

\begin{theorem}\label{thm:polyaballD}
The Dirichlet P\'{o}lya's conjecture for the unit ball  holds in any dimension $d\ge 2$, that is we have
\[
\N^\Dir_{\ball{d}}(\lambda)<W_d(\lambda)
\]
for all $\lambda>0$.
\end{theorem}

Our results in the Neumann case are restricted to the case $d=2$. Higher-dimensional Neumann problems are harder, and we intend to treat them in a subsequent paper.

We first state
\begin{lemma}\label{lem:Lambda0} 
The Neumann  P\'{o}lya's conjecture for $\disk=\ball{2}$ is valid for all $\lambda\in\left[0, \Lambda_0\right]$, where 
\[
\Lambda_0:=2\sqrt{3}.
\]
\end{lemma}

\begin{proof} 
Taking the span of $\{1, x, y\}$ as a test space  in the Rayleigh quotient for the Neumann Laplacian on $\disk$ gives $\mu_3(\disk)\le 4$. Therefore,
\[
\N^\Neu_\disk(\lambda)\quad \ge \quad 
\begin{cases}
1,&\quad \lambda\in[0,2),\\
3,&\quad \lambda\ge 2,
\end{cases}\quad\ge\quad  \frac{\lambda^2}{4}\qquad\text{for }\lambda\in[0,\Lambda_0].
\]
\end{proof}

We then prove
\begin{theorem}\label{thm:polyadiskN}
The Neumann P\'{o}lya's conjecture for the unit disk holds for all 
\begin{equation}\label{eq:Lambda1}
\lambda\ge \Lambda_1:=\frac{6\pi}{3\pi-8}.
\end{equation}
\end{theorem}

We note that $\Lambda_0>3$ and $\Lambda_1<14$, so we already have the validity of the Neumann P\'olya's conjecture for the disk for all $\lambda$ outside the interval $\left(3, 14\right)$. 

\begin{theorem}\label{thm:polyadiskN2}
The Neumann P\'{o}lya's conjecture for the unit disk holds for all $\lambda\in[3,14]$.
\end{theorem}

The proof of Theorem \ref{thm:polyadiskN2} is \emph{rigorous} but \emph{computer-assisted}. 
More specifically, it is based on a realisation of an algorithm which satisfies two fundamental 
principles.
\begin{description}
\item[Principle 1.] The algorithm should complete in a finite number of steps.
\item[Principle 2.] The algorithm should operate only with integer or rational numbers, thus avoiding \emph{any} use of floating-point arithmetic and \emph{any} rounding errors.
\end{description}

The combination of Lemma \ref{lem:Lambda0} and Theorems \ref{thm:polyadiskN} and \ref{thm:polyadiskN2} ensures that the Neumann P\'olya conjecture for the disk is valid  for all $\lambda>0$, that is we have
\begin{cor}
$\N^\Neu_{\disk}(\lambda)>\frac{\lambda^2}{4}$ for all $\lambda>0$.
\end{cor}

\begin{remark} Since P\'{o}lya's conjecture is scale-invariant, its validity for a unit ball immediately implies that it is valid for any ball of the same dimension.
\end{remark}

We additionally have the following generalisation of P\'{o}lya's result for tiling domains: we show that P\'olya's conjecture holds not only for domains which tile Euclidean space, but also for domains which tile another domain for which it is known to be true.
\begin{theorem}\label{thm:tiling}
Let $\Omega\subset\mathbb{R}^d$ be a domain for which either the Dirichlet or the Neumann P\'{o}lya's conjecture holds, and let $\Omega'$ be a domain which tiles $\Omega$. Then the same P\'{o}lya's conjecture also holds for $\Omega'$.
\end{theorem}  
\begin{proof}
Assume that $\Omega$ can be tiled by $\ell\ge 2$ congruent copies of $\Omega'$, so that $|\Omega|_d=\ell|\Omega'|_d$.  We have, by bracketing and since the eigenvalues of all the congruent copies coincide with those of $\Omega'$,
\[
\ell  \N^\Dir_{\Omega'}(\lambda)\le \N^\Dir_\Omega(\lambda)< \N^\Neu_\Omega(\lambda)\le \ell  \N^\Neu_{\Omega'}(\lambda).
\]
Assuming now \eqref{eq:polya} for all $\lambda\ge 0$, we get
\[
\ell  \N^\Dir_{\Omega'}(\lambda)\le\WeylC{d} |\Omega|_d \lambda^d=\WeylC{d} \ell |\Omega'|_d \lambda^d\le \ell  \N^\Neu_{\Omega'}(\lambda),
\]
and the result follows by cancelling $\ell$.
\end{proof}

\begin{remark} If the inequalities in P\'olya's conjecture \eqref{eq:polya} for $\Omega$ are strict, they are also strict for $\Omega'$. 
\end{remark}

Theorem \ref{thm:tiling} immediately implies the following  
\begin{cor}\label{cor:polyacone}
Let $\widetilde{\Omega}\subset\mathbb{S}^{d-1}$ be a spherical domain which tiles $\mathbb{S}^{d-1}$.  Then the Dirichlet P\'{o}lya's conjecture holds for the spherical cone in $\R^d$ with the base $\widetilde{\Omega}$ and the vertex at the origin.
\end{cor}

In the planar case we get
\begin{cor}\label{cor:polyasector}
P\'{o}lya's conjecture holds for any circular sector $S_\alpha$ with an aperture $\alpha=\frac{2\pi}{\ell}$, where $\ell\in\{2,3,\dots\}$.
\end{cor}
We refer also to  \cite{Fre22} for an alternative proof of Corollary \ref{cor:polyasector} for sufficiently large (but unspecified) $\ell$.

We can in fact extend the result of Corollary \ref{cor:polyasector} to arbitrary sectors.
\begin{theorem}\label{thm:polyasector}
P\'{o}lya's conjecture holds for any circular sector $S_\alpha$ with an aperture $\alpha\in(0,2\pi]$, that is
\[
\N^\Dir_{S_\alpha}(\lambda) < \frac{\alpha\lambda^2}{8\pi} < \N^\Neu_{S_\alpha}(\lambda)
\]
for all $\lambda>0$.
\end{theorem}

\begin{remark} The result of Theorem \ref{thm:polyasector} in the case $\alpha=2\pi$ (the disk with the radial slit) follows immediately from Theorems \ref{thm:polyaballD} ($d=2$) and \ref{thm:polyadiskN} by Dirichlet--Neumann bracketing.
\end{remark}

\subsection*{Plan of the paper}  In the next section we describe two lattice counting problems \eqref{eq:PtD} and \eqref{eq:PtN}, variants of which were originally introduced by N. Kuznetsov and B. Fedosov in \cite{kufe},  and which are  closely linked to the Dirichlet and Neumann eigenvalue counting problems in the ball.  The key novel tool  is Theorem \ref{thm:sher},  originally obtained in part in \cite{Sher}, which gives a {\em uniform}  bound between the eigenvalue and the lattice counting functions, as opposed to asymptotic relations that were previously known. We provide  an independent  proof of this result in \S\ref{sec:thmpf}. In \S\ref{sec:analytic} we state the results on the lattice counting functions which are sufficient for proving P\'olya's conjecture for balls. The bulk of the paper, \S\S\ref{sec:proofD2}--\ref{sec:Neumann2computer}, is devoted to the proofs of these results. Theorem \ref{thm:polyasector} is proved in \S\ref{sec:sectors}. 

\subsection*{Acknowledgements} The authors would like to thank L. Friedlander,  R. Laugesen, Z. Rudnick, and I. Wigman for useful discussions. We are also very grateful to the anonymous referee  for numerous helpful suggestions. Research of NF was supported by the grant No. 22-11-00092 of the Russian Science Foundation. Research of ML was partially supported by the EPSRC grants EP/W006898/1 and EP/V051881/1, and by the University of Reading RETF Open Fund.  Research of IP was partially supported by NSERC and FRQNT. Research of DS was partially supported by an FSRG from DePaul University.

\section{Dirichlet and Neumann eigenvalues of the ball and lattice counting problems}\label{sec:evdisk}
Throughout this paper, with $\nu\ge 0$, and $k\in\mathbb{N}$, let $J_\nu(z)$ be the Bessel functions of order $\nu$,  let $j_{\nu,k}$ be the $k$th positive zero of $J_\nu$, and let $j'_{\nu, k}$ be  the $k$th positive zero of its derivative $J'_\nu$, with the exception of $J'_0$ for which $j'_{0,1}=0$. 

It is well known that the eigenvalues of the Dirichlet Laplacian in the unit ball are given by the squares of the zeros of the cylindrical Bessel functions. Namely, considering the Dirichlet Laplacian in $\ball{d}$, we have the simple eigenvalues 
\[
\lambda_{d,0,k}=\left(j_{d/2-1,k}\right)^2,\qquad k\in\mathbb{N},
\]
that is of multiplicity
\[
\kappa_{d, 0}:=1,
\]
and the eigenvalues 
\[
\lambda_{d,m,k}=\left(j_{m+d/2-1,k}\right)^2,\qquad m, k\in\mathbb{N},
\]
of multiplicity 
\[
\kappa_{d, m}:=\binom{m+d-1}{d-1}-\binom{m+d-3}{d-1}. 
\]

\begin{remark} We note that the numbers $\kappa_{d, m}$, $d\ge 2$, $m\in \mathbb{N}_0=\mathbb{N}\cup\{0\}$, coincide with the multiplicity of the eigenvalue $m(m+d-2)$ of the Laplace--Beltrami operator on the unit sphere $\mathbb{S}^{d-1}$, or alternatively with the dimension of the space of homogeneous harmonic polynomials  of degree $m$ in $\mathbb{R}^d$.  In the planar case we have
\[
\kappa_{2, m}=2\qquad\text{for }m\in\mathbb{N}.
\]
\end{remark}

We therefore have
\begin{equation}\label{eq:NDball}
\N^\Dir_{\ball{d}}(\lambda)=\sum_{m=0}^\infty \kappa_{d,m}\#\left\{k\in\mathbb{N}: j_{m+d/2-1,k}\le \lambda\right\}.
\end{equation}

\begin{remark}\label{rem:NDballfinite} The sum in \eqref{eq:NDball} is in fact finite: we have\footnote{Here and further on, $\entire{x}=\max\{k\in\mathbb{Z}: k\le x\}$ denotes the integer part of $x\in\mathbb{R}$, and $\ceiling{x}=\min\{k\in\mathbb{Z}: k\ge x\}$ denotes its ceiling.}
\begin{equation}\label{eq:NDballfinite}
\N^\Dir_{\ball{d}}(\lambda)=\sum_{m=0}^{\entire{\lambda-d/2+1}}\kappa_{d,m}\#\left\{k\in\mathbb{N}: j_{m+d/2-1,k}\le \lambda\right\},
\end{equation}
This is due to the fact that $j_{\nu,1}>\nu$ \cite[Eq. 10.21.3]{dlmf}. Note that in \eqref{eq:NDballfinite} and further on, any sum in which the lower limit exceeds the upper limit is assumed to be zero, which immediately gives $\N^\Dir_{\ball{d}}(\lambda)=0$ for $\lambda<\frac{d}{2}-1$.
\end{remark}

In the planar case the expression \eqref{eq:NDball} simplifies to 
\[
\N^\Dir_\disk(\lambda)=\#\left\{k\in\mathbb{N}: j_{0,k}\le \lambda\right\}+2\sum_{m=1}^{\entire{\lambda-d/2+1}}\#\left\{(m,k)\in\mathbb{N}^2: j_{m,k}\le \lambda\right\}.
\]

Similarly, the eigenvalues of the Neumann Laplacian in the unit disk $\disk$ are given by the squares of the zeros of the derivatives of  Bessel functions. We have the simple eigenvalues 
\[
\mu_{0,k}=\left(j'_{0,k}\right)^2,\qquad k\in\mathbb{N},
\]
and the double eigenvalues  
\[
\mu_{m,k}=\left(j'_{m,k}\right)^2,\qquad m, k\in\mathbb{N}.
\]
We therefore have
\begin{equation}\label{eq:NNdisk}
\N^\Neu_\disk(\lambda)=\#\left\{k\in\mathbb{N}: j'_{0,k}\le \lambda\right\}+2\sum_{m=1}^{\entire{\lambda-d/2+1}}\#\left\{(m,k)\in\mathbb{N}^2: j'_{m,k}\le \lambda\right\},
\end{equation} 
where the sum is again finite since $j'_{\nu,1}\ge \nu$  \cite[Eq. 10.21.3]{dlmf}.

For illustrative purposes only, we show the graphs of the Dirichlet and Neumann eigenvalue counting functions for the disk in Figure \ref{fig:NDplot}.

\begin{figure}[htpb]
\centering
\includegraphics{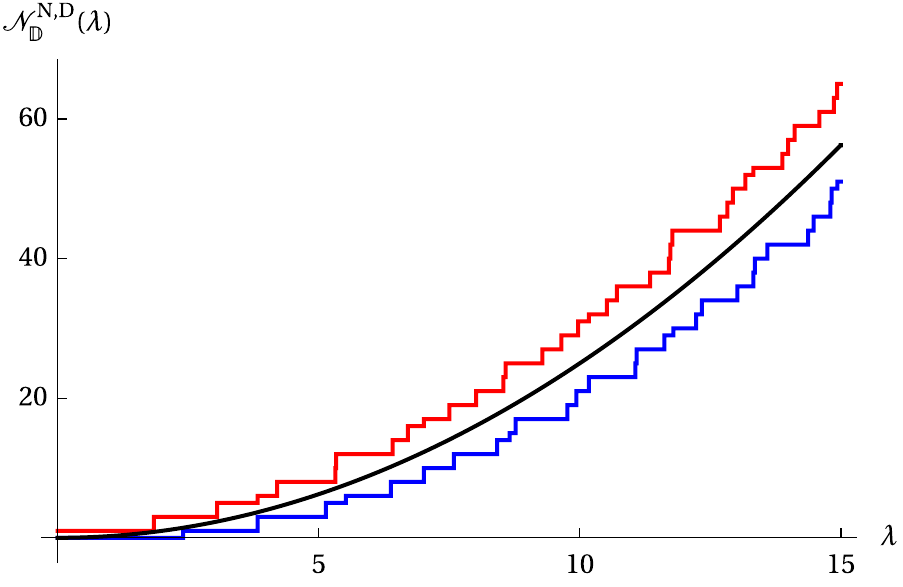}
\caption{The Dirichlet eigenvalue counting function $\N^\Dir_{\disk}(\lambda)$ (blue), the Neumann eigenvalue counting function $\N^\Neu_{\disk}(\lambda)$ (red), and the leading Weyl's term $W_d(\lambda)=\frac{\lambda^2}{4}$ (black) in dimension $d=2$. The plot is produced using the floating-point evaluation of zeros of the Bessel functions and their derivatives. If we were to assume (contrary to the philosophy of this paper) the validity of floating-point arithmetic, this plot would have presented a \emph{numerically assisted} (as opposed to \emph{computer-assisted)} ``proof'' of P\'{o}lya's conjecture for the disk for $\lambda\lessapprox 15$.\label{fig:NDplot}}
\end{figure}

We will be comparing the  counting functions $\N^\Dir_{\ball{d}}(\lambda)$ and $\N^\Neu_{\disk}(\lambda)$ with some weighted lattice counting functions. 
Let 
\[
h(x):=\frac{1}{\pi}\left(\sqrt{1-x^2}-x\arccos x\right), \qquad x\in[0,1],
\]
and let $\Ph$ be a planar region under the graph of $h(x)$,
\[
\Ph:=\left\{(x, y): x\in[0,1], y\in[0, h(x)]\right\}.
\]

Let, for $\lambda>0$, 
\begin{equation}\label{eq:Glambda}
G_\lambda(z):=\lambda h\left(\frac{z}{\lambda}\right)=\frac{1}{\pi}\left(\sqrt{\lambda^2-z^2}-z\arccos\frac{z}{\lambda}\right),\qquad z\in[0,\lambda],
\end{equation}
and let $\Ph_\lambda$ be a dilation of $\Ph$ with coefficient $\lambda$ with respect to the origin, 
\begin{equation}\label{eq:Phlambda}
\Ph_\lambda= \left\{(z,y): 0\le z\le\lambda, 0\le y\le G_\lambda(z)\right\},
\end{equation}
that is, the region under the graph of $G_\lambda(z)$.

Let 
\[
Q^\Dir_d(\lambda):=\left\{(m,k)+\left(\frac{d}{2}-1, -\frac{1}{4}\right)\in \Ph_\lambda: (m,k)\in \mathbb{N}_0\times \mathbb{N}\right\} 
\]
and 
\[
Q^\Neu_2(\lambda):=\left\{(m,k)+\left(0, -\frac{3}{4}\right)\in \Ph_\lambda: (m,k)\in \mathbb{N}_0\times \mathbb{N}\right\} 
\]
be the sets of shifted integer lattice points which lie in $\Ph_\lambda$, see Figure \ref{fig:2}. The definitions of the two sets for $d=2$ differ by a vertical shift. The reason for choosing this particular notation will become evident later. 

We now introduce the \emph{weighted lattice point counting functions}
\begin{equation}\label{eq:PtD}
\Pt^\Dir_d(\lambda):=\sum_{\left(m+\frac{d}{2}-1,k-\frac{1}{4}\right)\in Q^\Dir_d(\lambda)}\kappa_{d,m}
\end{equation}
and
\begin{equation}\label{eq:PtN}
\Pt^\Neu_2(\lambda):=\sum_{\left(m,k-\frac{3}{4}\right)\in Q^\Neu_2(\lambda)}\kappa_{2,m}.
\end{equation}
It is immediately seen from the definitions \eqref{eq:Phlambda}--\eqref{eq:PtN} that with $\aleph\in\{\Dir,\Neu\}$ we have
\begin{equation}\label{eq:PtDN}
\Pt^\aleph_d(\lambda)=\sum_{m=0}^{\entire{\lambda-d/2+1}}\kappa_{d,m}\entire{G_\lambda\left(m+\frac{d}{2}-1\right)+s^\aleph},
\end{equation}
where
\begin{equation}\label{eq:saleph}
s^\Dir:=\frac{1}{4}, \qquad s^\Neu:=\frac{3}{4}.
\end{equation}
 
\begin{figure}[htpb]
\centering
\includegraphics{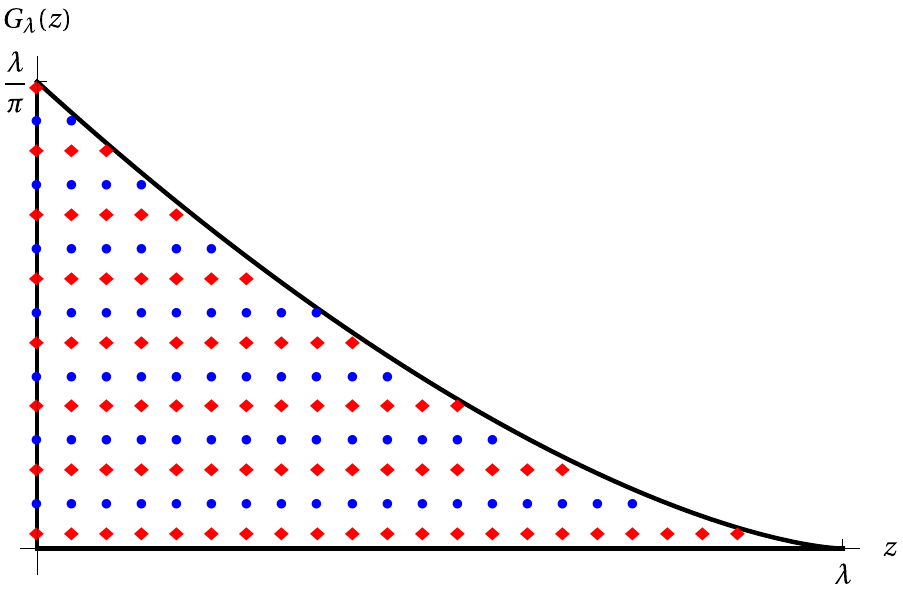}
\caption{The region $\Ph_\lambda$, and the sets of shifted lattice points $Q^\Dir_d(\lambda)$ (blue disks) and  $Q^\Neu_d(\lambda)$ (red diamonds), shown here for $d=2$ and $\lambda=23$.\label{fig:2}}
\end{figure}

It is well known that as $\lambda\to+\infty$, the asymptotics  of the lattice point counting function  $\Pt_d^\Dir(\lambda)$ is intricately linked to the asymptotics of the eigenvalue counting function $\N^\Dir_{\ball{d}}(\lambda)$. This was first shown in the planar case in \cite{kufe} and later re-discovered in \cite{cdv}, see also \cite{gravel}. Namely, in some appropriate sense,
\[
\N^\Dir_{\ball{d}}(\lambda)\sim \Pt^\Dir_d(\lambda)\qquad\text{as }\lambda\to+\infty.
\]
This observation, together with asymptotic bounds on the difference between the two functions, has been used to great effect to estimate the remainder in Weyl's law for the unit ball. In particular, for the Dirichlet problem in the disk the two-term Weyl asymptotics \eqref{eq:twoterm2d}  holds  with an improved remainder estimate 
\[
O\left(\lambda^{131/208}(\log\lambda)^{18627/8320}\right),
\] 
see \cite{GMWW} (the remainder estimate  $O\left(\lambda^{2/3}\right)$ was already obtained in \cite{kufe}, \cite{cdv}). Similar improved remainder estimates are also known in the Dirichlet case for higher-dimensional balls \cite{Guo21} and in the planar Neumann case \cite{GWW}. 

As has been recently found in \cite{Sher} in the Dirichlet case, there is a further simple \emph{non-asymptotic} relation between the lattice point and the eigenvalue counting functions, which lies at the cornerstone of our proofs of Theorems  \ref{thm:polyaballD} and \ref{thm:polyadiskN}. 

\begin{theorem}\label{thm:sher} For any $d\ge 2$ and any $\lambda\ge 0$, we have
\[
\N^\Dir_{\ball{d}}(\lambda)\le \Pt_d^\Dir(\lambda).
\]
We also have, for any $\lambda\ge 0$,
\[
\Pt_2^\Neu(\lambda)\le \N^\Neu_{\disk}(\lambda).
\]
\end{theorem}

\section{{Proof of Theorem \ref{thm:sher}}}
\label{sec:thmpf}

We start by introducing some additional notation. Set, for $\nu\ge 0$, $\lambda\ge 0$,  and $\aleph\in\{\Dir,\Neu\}$,
\begin{equation}\label{eq:Ams}
A_\nu^\aleph(\lambda):=
\begin{cases}
G_\lambda(\nu)+s^\aleph,\quad&\text{if }\lambda\ge\nu,\\
s^\aleph,\quad&\text{if }0\le \lambda<\nu,
\end{cases}
\end{equation}
where $G_\lambda$ is defined by \eqref{eq:Glambda} and $s^\aleph$ is defined by \eqref{eq:saleph}.
Some typical graphs of the functions $A_\nu^\aleph(\lambda)$ are shown in Figure \ref{fig:3}.

The crucial step in the proof of Theorem \ref{thm:sher} comes from the following bounds on the number of zeros of Bessel functions and their derivatives below a given number.
\begin{prop}\label{prop:mainbound} Let $\nu\ge 0$ and $\lambda\ge 0$. Then
\begin{equation}\label{eq:nucountD}
\#\left\{k\in\mathbb{N}: j_{\nu,k}\le \lambda\right\}\le \entire{A_{\nu}^\Dir(\lambda)}
\end{equation} 
and
\begin{equation}\label{eq:nucountN}
\#\left\{k\in\mathbb{N}: j'_{\nu,k}\le \lambda\right\}\ge \entire{A_{\nu}^\Neu(\lambda)}.
\end{equation} 
\end{prop}

\begin{remark} For $\lambda\in[0,\nu]$, the inequalities \eqref{eq:nucountD} and \eqref{eq:nucountN} become the trivial identities $0=\entire{s^\aleph}=0$. 
\end{remark} 

\begin{remark} For $\nu=0$, the inequality \eqref{eq:nucountD} is equivalent to $j_{0,k}\ge \pi\left(k-\frac{1}{4}\right)$, which was proved in \cite{Het}.
\end{remark} 

\begin{proof}[Proof of Proposition \ref{prop:mainbound}] We recall the representations of the Bessel functions  of the first and second kind, $J_\nu$ and $Y_\nu$, and their derivatives in terms of the so-called \emph{modulus functions} $M_\nu$ and $N_\nu$ and the \emph{phase functions} $\theta_\nu$ and $\phi_\nu$,
\[
\begin{alignedat}{2}
J_{\nu}(x)&=M_{\nu}(x)\cos\theta_{\nu}(x),
&\quad Y_{\nu}(x)&=M_{\nu}(x)\sin\theta_{\nu}(x),\\
J'_{\nu}(x)&=N_{\nu}(x)\cos\phi_{\nu}(x),
&\quad Y'_{\nu}(x)&=N_{\nu}(x)\sin\phi_{\nu}(x)
\end{alignedat}
\]
(see \cite[Eqs. 10.18.4--5]{dlmf}). We will be using various properties of the phase functions below; for a review of these properties see \cite{Hor}. 
We will be only considering the cases $\nu\ge 0$ and $x\ge 0$ for which the moduli $M_\nu(x)$ and $N_\nu(x)$ are both positive. 

Let us concentrate first on \eqref{eq:nucountD}. We have $J_\nu(x_0)=0$ if and only if $\cos\theta_\nu(x_0)=0$, and so if and only if 
\begin{equation}\label{eq:thetaZ}
\frac{1}{\pi}\theta_\nu(x_0)+\frac{1}{2}\in\mathbb{Z}.
\end{equation}
Note that the phase function $\theta_\nu(x)$ satisfies $\theta_\nu(x)\to-\frac{\pi}{2}$ as $x\to+0$ \cite[Eq. 10.18.3]{dlmf}, and that it is monotone increasing for $x\in(0,+\infty)$ \cite[Theorem 1]{Hor}, therefore \eqref{eq:thetaZ} can be replaced by 		
\[
B^\Dir_\nu(x_0)\in\mathbb{N},
\] 
where
\begin{equation}\label{eq:Bdir}
B^\Dir_\nu(x):=\frac{1}{\pi}\theta_\nu(x)+\frac{1}{2},
\end{equation}
and therefore
\begin{equation}\label{eq:numberDasB}
\#\left\{k\in\mathbb{N}: j_{\nu,k}\le \lambda\right\}=\entire{B^\Dir_\nu(\lambda)}.
\end{equation}

We have
\begin{equation}\label{eq:ftildeasympt} 
G_\lambda(\nu) = \frac{\lambda}{\pi}-\frac{\nu}{2}+O\left(\frac{1}{\lambda}\right)\qquad\text{as }\lambda\to+\infty.
\end{equation}
Using the asymptotics\footnote{See also \cite{heitman} for all the coefficients of the full asymptotic expansion and some useful remarks.}  \cite[Eq. (21)]{Hor},  \cite[Eq. 10.18.18]{dlmf},
\[
\theta_\nu(\lambda)=\lambda-\frac{\pi}{2}\left(\nu+\frac{1}{2}\right)+O\left(\frac{1}{\lambda}\right)\qquad\text{as }\lambda\to+\infty,
\] 
and \eqref{eq:Bdir}, \eqref{eq:ftildeasympt}, we obtain
\[
B^\Dir_\nu(\lambda)-A^\Dir_\nu(\lambda)\to 0\qquad\text{as }\lambda\to+\infty.
\]
Further,
\[
\frac{\dr B^\Dir_\nu(\lambda)}{\dr \lambda}=\frac{1}{\pi}\theta'_\nu(\lambda)>\frac{\sqrt{\lambda^2-\nu^2}}{\pi\lambda}=\frac{\dr A^\Dir_\nu(\lambda)}{\dr \lambda}
\]
for $\lambda\ge \nu$ by \cite[Eq. (56)]{Hor}. As we additionally have
\[
\frac{\dr B^\Dir_\nu(\lambda)}{\dr \lambda}>0=\frac{\dr A^\Dir_\nu(\lambda)}{\dr \lambda}
\]
for $\lambda\in(0,\nu]$, the function $B^\Dir_\nu(\lambda)-A^\Dir_\nu(\lambda)$ is monotone increasing on $(0,+\infty)$ and tends to zero at infinity. Thus, we have proved that
\begin{equation}\label{eq:BminusAD}
B^\Dir_\nu(\lambda)<A^\Dir_\nu(\lambda)\qquad\text{for } \lambda\in[0,+\infty).
\end{equation}
Combining \eqref{eq:numberDasB} and \eqref{eq:BminusAD}  proves  \eqref{eq:nucountD}.

We now prove \eqref{eq:nucountN}. In the same manner we have $J'_\nu(x_0)=0$ if and only if 
\begin{equation}\label{eq:phiZ}
\frac{1}{\pi}\phi_\nu(x_0)+\frac{1}{2}\in\mathbb{Z}.
\end{equation}
We note that  the phase function $\phi_\nu$ satisfies $\phi_\nu(x)\to\frac{\pi}{2}$ as $x\to+0$ \cite[Eq. 10.18.3]{dlmf}. Also, $\phi_\nu(x)$ is monotone increasing for $x\in(\nu,+\infty)$  and monotone decreasing for $x\in(0,\nu)$ \cite[Theorem 1]{Hor}, with $\phi_\nu(\nu)>-\frac{\pi}{2}$ \cite[formula (60)]{Hor}. Thus, the condition \eqref{eq:phiZ} can be replaced by   
\[
B^\Neu_\nu(x_0)\in\mathbb{N},
\] 
where
\begin{equation}\label{eq:BNeu}
B^\Neu_\nu(x):=\frac{1}{\pi}\phi_\nu(x)+\frac{1}{2},
\end{equation}
and therefore
\begin{equation}\label{eq:numberNasB}
\#\left\{k\in\mathbb{N}: j'_{\nu,k}\le \lambda\right\}=\entire{B^\Neu_\nu(\lambda)}.
\end{equation}

Using the asymptotics \cite[Eq. (22)]{Hor},  \cite[Eq. 10.18.21]{dlmf},
\[
\phi_\nu(\lambda)=\lambda-\frac{\pi}{2}\left(\nu-\frac{1}{2}\right)+O\left(\frac{1}{\lambda}\right)\qquad\text{as }\lambda\to+\infty,
\] 
and \eqref{eq:BNeu},  \eqref{eq:ftildeasympt}, we get
\[
B^\Neu_\nu(\lambda)-A^\Neu_\nu(\lambda)\to 0\qquad\text{as }\lambda\to+\infty.
\]  
Also, 
\[
\frac{\dr B^\Neu_\nu(\lambda)}{\dr\lambda}=\frac{1}{\pi}\phi'_\nu(\lambda)<\frac{\sqrt{\lambda^2-\nu^2}}{\pi \lambda}=\frac{\dr A^\Neu_\nu(\lambda)}{\dr \lambda}
\]
for $\lambda\ge \nu$ by \cite[formula following Eq. (58)]{Hor}. As we additionally have 
\[
\frac{\dr B^\Neu_\nu(\lambda)}{\dr \lambda}<0=\frac{\dr A^\Neu_\nu(\lambda)}{\dr \lambda}
\]
for $\lambda\in(0,\nu]$, the function $B^\Neu_\nu(\lambda)-A^\Neu_\nu(\lambda)$ is monotone decreasing on $(0,+\infty)$ and tends to zero at infinity. Thus, we have proved that
\begin{equation}\label{eq:BminusAN}
B^\Neu_\nu(\lambda)>A^\Neu_\nu(\lambda)\qquad\text{for } \lambda\in[0,+\infty).
\end{equation}
Combining \eqref{eq:numberNasB} and \eqref{eq:BminusAN}  proves  \eqref{eq:nucountN}.
\end{proof}

We illustrate inequalities \eqref{eq:BminusAD} and \eqref{eq:BminusAN} in Figure \ref{fig:3}.

\begin{figure}[htpb]
\centering
\includegraphics{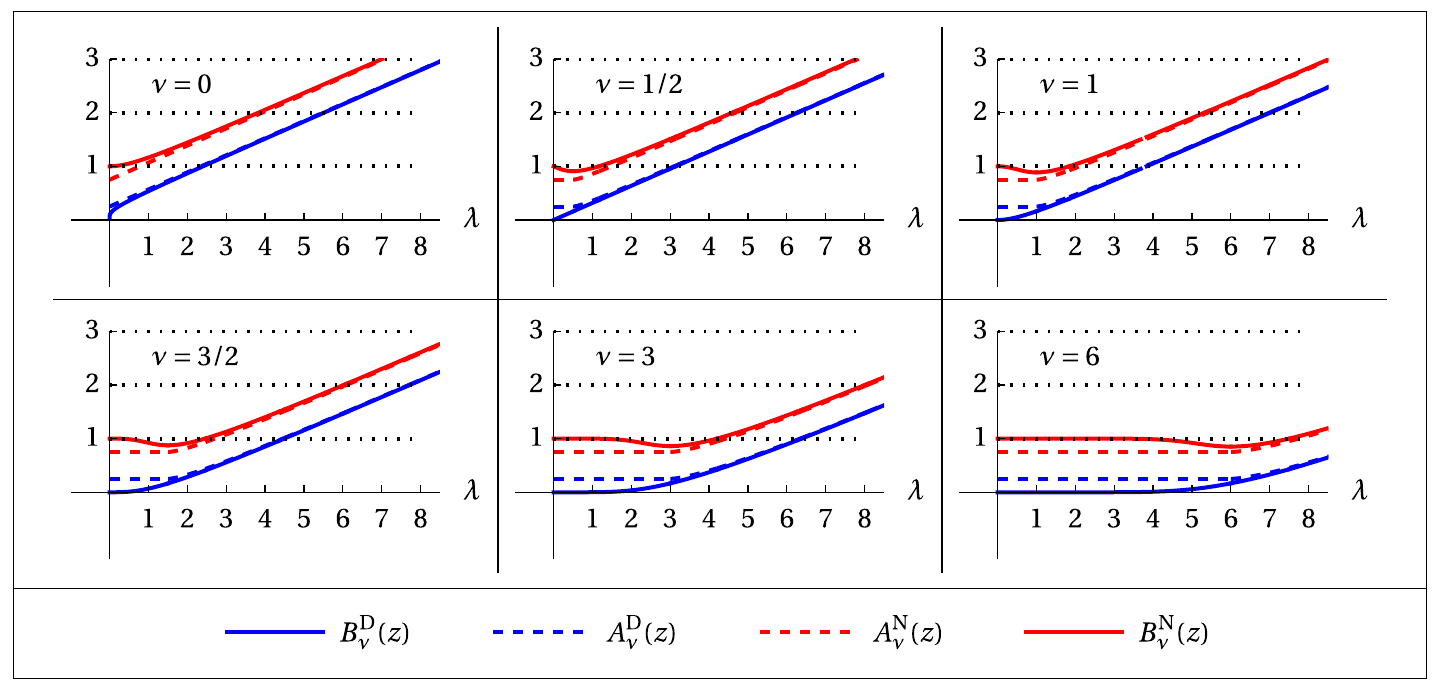}
\caption{An illustration of inequalities \eqref{eq:BminusAD} and \eqref{eq:BminusAN}. The plots of $B^\Dir_\nu(\lambda)$ and $B_\nu^\Neu(\lambda)$ are drawn using the recipe from \cite{Hor}. We remark that $B^\Neu_\nu(\lambda)$ has a minimum at $\lambda=\nu$. \label{fig:3}}
\end{figure}

The bound \eqref{eq:nucountD} can be also proved without relying on the properties of the Bessel phase function. Instead, one uses the known asymptotics of the Bessel zeros and the Sturm comparison theorem. We present this alternative argument below for an interested reader. The bound \eqref{eq:nucountN} can be proved in the same manner; we omit the details.

Let $\nu\ge 0$, and let us consider the function
\[
\sin\left(\pi A^\Dir_\nu(x)\right).
\]
Let $a_{\nu, k}$ be its $k$th  zero in $[\nu,+\infty)$, ordered increasingly. Obviously, 
\begin{equation}\label{eq:counta}
\#\left\{k\in\mathbb{N}: a_{\nu,k}\le\lambda\right\}=\entire{A^\Dir_\nu(\lambda)}.
\end{equation}
We will prove
\begin{lemma}\label{lem:altproof} 
$a_{\nu, k}\le j_{\nu,k}$ for all $k\in\mathbb{N}$.
\end{lemma}
Together with \eqref{eq:counta}, Lemma \ref{lem:altproof} immediately implies  \eqref{eq:nucountD}.

\begin{proof}[Proof of Lemma \ref{lem:altproof}] The function $U_\nu(x):=\sqrt{x}J_\nu(x)$ satisfies the differential equation
\begin{equation}\label{eq:a11}
U''_\nu(x) + \left(1 - \frac{\nu^2-1/4}{x^2}\right) U_\nu(x) = 0.
\end{equation}
Consider the function
\[
V_\nu(x) := \frac{\sqrt{x}}{(x^2-\nu^2)^{1/4}} \sin \left(\pi A^\Dir_\nu(x) + b\right)
\]
with some $b \in \left[0, \frac{\pi}{4}\right)$.
Then
\begin{align*}
V'_\nu(x) &= - \frac{\nu^2}{2 \sqrt x \left(x^2-\nu^2\right)^{5/4}} \sin\left(\pi A^\Dir_\nu(x) + b\right)
+ \frac{\left(x^2-\nu^2\right)^{1/4}}{\sqrt x}\cos\left(\pi A^\Dir_\nu(x) + b\right),
\\
V''_\nu(x) &= \left( \frac{6\nu^2x^2 - \nu^4}{4 x^{3/2} (x^2-\nu^2)^{9/4}} 
- \frac{(x^2-\nu^2)^{3/4}}{x^{3/2}} \right) \sin\left(\pi A^\Dir_\nu(x) + b\right) .
\end{align*}
Therefore,
\begin{equation}\label{eq:a12}
V''_\nu(x) + 
\left(1 - \frac{\nu^2}{x^2} - \frac{\nu^2(6x^2-\nu^2)}{4x^2(x^2-\nu^2)^2} \right) V_\nu(x) = 0
\qquad \text{for }  x\in(\nu, +\infty).
\end{equation}
Denote by $v_{\nu,k}= v_{\nu,k}(b)$ the $k$th zero of the function $V_\nu$.
By the definitions of $V_\nu$ and $A_\nu^\Dir$,
\[
\sqrt{v_{\nu,k}^2-\nu^2} - \nu \arccos\frac{\nu}{v_{\nu,k}} +  \frac{\pi}{4} + b  = \pi k.
\]
As
\[
\sqrt{v_{\nu,k}^2-\nu^2} = v_{\nu,k} + O\left(v_{\nu,k}^{-1}\right) 
\qquad \text{and}\qquad \arccos\frac{\nu}{v_{\nu,k}} = \frac\pi2 +  O\left(v_{\nu,k}^{-1}\right)  \qquad\text{as } k \to \infty,
\]
we have
\[
v_{\nu,k} (b) = \pi \left(k + \frac{\nu}{2}-\frac{1}{4}\right) - b + O\left(k^{-1}\right) \qquad\text{as } k \to \infty.
\]
On the other hand the asymptotics
\[
j_{\nu,k} = \pi \left(k + \frac\nu2 - \frac14\right) + O\left(k^{-1}\right) \qquad\text{as } k \to \infty.
\]
is well known, see for example \cite[Eq. 10.21.19]{dlmf}.

Suppose that $b>0$. Then there exists $K\in\mathbb{N}$ such that
\begin{equation}\label{eq:a13}
j_{\nu,k} > v_{\nu,k}(b)\qquad\text{for }k\ge K.
\end{equation}

The coefficient in front of $U_\nu$ in \eqref{eq:a11} is greater than 
the coefficient in front of $V_\nu$ in \eqref{eq:a12}:
\[
1 - \frac{\nu^2-1/4}{x^2} > 1 - \frac{\nu^2}{x^2}
> 1 - \frac{\nu^2}{x^2} - \frac{\nu^2(6x^2-\nu^2)}{4x^2(x^2-\nu^2)^2}
\qquad\text{for all } x \in [\nu, +\infty) .
\]
By the Sturm comparison theorem there is 
a zero of $U_\nu$ between $v_{\nu,k}(b)$ and $v_{\nu,k+1}(b)$.
So, if $j_{\nu, k_0} \le v_{\nu, k_0}(b)$ for some number $k_0$,
then $j_{\nu, k_0+1} \le v_{\nu, k_0+1}(b)$, and by induction
$j_{\nu, k} \le v_{\nu, k}(b)$ for all $k \ge k_0$ which contradicts \eqref{eq:a13}.
Therefore, \eqref{eq:a13} holds for all natural $k$.

Finally, each function $v_{\nu, k}(b)$ is continuous in $b$.
Thus, 
\[
j_{\nu, k} \ge v_{\nu, k}(0) = a_{\nu,k}.
\]
\end{proof}

Returning now to the proof of Theorem \ref{thm:sher}, we rewrite \eqref{eq:PtDN} as 
\begin{equation}\label{eq:PtA}
\Pt_d^\aleph(\lambda)=\sum_{m=0}^{\entire{\lambda-d/2+1}} \kappa_{d,m}\entire{A_{m+d/2-1}^\aleph(\lambda)}.
\end{equation}
Theorem  \ref{thm:sher} now immediately follows from  Proposition \ref{prop:mainbound} with account of \eqref{eq:PtA}, \eqref{eq:NDballfinite}, and \eqref{eq:NNdisk}.

\section{From the weighted lattice point count towards P\'{o}lya's conjecture}\label{sec:analytic}

By Theorem   \ref{thm:sher}, the Dirichlet P\'{o}lya's conjecture for $\ball{d}$ would follow immediately if we can prove that 
\begin{equation}\label{eq:Pbounds}
 \Pt_d^\Dir(\lambda)<W_d(\lambda)
\end{equation}
for all $\lambda\in(0,+\infty)$, where $W_d(\lambda)$ is defined by \eqref{eq:Wd}.

Similarly, the Neumann P\'{o}lya's conjecture for $\disk$  would follow immediately if we can prove that 
\begin{equation}\label{eq:PboundsN}
 \Pt_2^\Neu(\lambda)>W_2(\lambda)
\end{equation}
for all $\lambda\in(\Lambda_0,+\infty)$; we note that we have already dealt with $\lambda\le\Lambda_0$ by Lemma \ref{lem:Lambda0}.

We establish \eqref{eq:Pbounds} in the following cases, which will be dealt with separately.

\begin{theorem}\label{thm:count1}
The inequality  
\begin{equation}\label{eq:ptD2}
\Pt_2^\Dir(\lambda)<W_2(\lambda)=\frac{\lambda^2}{4}
\end{equation}
holds for all $\lambda>0$.
\end{theorem}
Theorem \ref{thm:count1} will be proved in \S\ref{sec:proofD2}. Together with Theorem \ref{thm:sher}, it implies Theorem \ref{thm:polyaballD} in the planar case.

\begin{theorem}\label{thm:count2}
The inequalities  \eqref{eq:Pbounds} 
hold for all $d\ge 3$ and $\lambda>0$.
\end{theorem}
Theorem \ref{thm:count2} will be proved in \S\ref{sec:proofDd}. Together with Theorem \ref{thm:sher}, it implies Theorem \ref{thm:polyaballD} for higher-dimensional balls.

In the Neumann case, the situation is more delicate, as we cannot expect \eqref{eq:PboundsN} to hold for all values of $\lambda\in(0,+\infty)$ since $\Pt_2^\Neu(\lambda)$ is identically zero for $\lambda<\frac{\pi}{4}$, see Figure \ref{fig:experiments}.

\begin{figure}[htpb]
\centering
\includegraphics{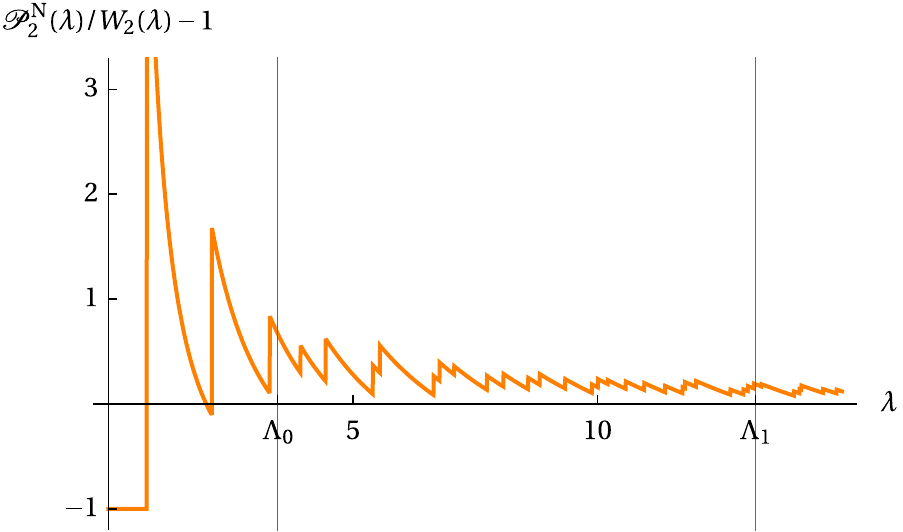}
\caption{A numerical experiment: the computed $\Pt_2^\Neu(\lambda)/W_2(\lambda)-1$ as a function of $\lambda$.
\label{fig:experiments}}
\end{figure}

We prove the following results.

\begin{theorem}\label{thm:count3}
The inequality
\[
\Pt_2^\Neu(\lambda)>W_2(\lambda)=\frac{\lambda^2}{4}
\]
holds for all $\lambda\ge \Lambda_1$, where $\Lambda_1$ is given by \eqref{eq:Lambda1}. 
\end{theorem}
Theorem \ref{thm:count3} will be proved in \S\ref{sec:proofN2}. Together with Theorem \ref{thm:sher}, it implies Theorem \ref{thm:polyadiskN}.

We are further able to eliminate the remaining gap in the Neumann case.
\begin{theorem}\label{thm:count4}
The inequality \eqref{eq:PboundsN} holds for any $\lambda\in\left(\Lambda_0, \Lambda_1\right)$. 
\end{theorem}
The proof of this result, presented in \S\ref{sec:Neumann2computer}, is computer-assisted. Theorem \ref{thm:count4} implies Theorem \ref{thm:polyadiskN2}.

\medskip

In all cases, we deal with estimating a (weighted) count of (shifted) lattice points under the graph of a particular function $G_\lambda$. Such problems have been extensively studied in number theory, going back to the Gauss circle problem. Important  contributions in the general case can be traced through the works of van der Corput \cite{vdC} and Kr\"{a}tzel \cite{kratzel} to some very recent results of Laugesen and Liu \cite{liuthesis, lauliu18a}. In particular, \cite[Proposition 15]{lauliu18b} is directly applicable (with account of the fact that Laugesen and Liu do not count the points on the vertical axis and do not double-count the points inside) to our shifted lattice point count $\Pt_2^\Dir(\lambda)$, yielding the bound
\[
\Pt_2^\Dir(\lambda)\le \frac{\lambda^2}{4}+\left(\frac{2}{3}+\frac{1}{\pi}-\frac{\sqrt{3}}{2\pi}\right)\lambda\approx  \frac{\lambda^2}{4}+0.7093\lambda.
\] 
Unfortunately, since the coefficient in front of $\lambda$ in this formula is positive, this bound is weaker than our required bound \eqref{eq:ptD2}.
We need therefore to obtain sharper lattice point count bounds than those available generally, and to do so we additionally use some properties of the \emph{derivative} of the function    $G_\lambda$ in addition to the properties of the function itself, see Theorems \ref{thm:countDdisk} and \ref{thm:countNdisk},  and also Remarks \ref{rem:countDdisk} and \ref{rem:countNdisk} for an informal explanation.

For future use, we summarise below some elementary properties of the function $G_\lambda$.

The first lemma is checked by a direct calculation.
\begin{lemma}\label{lem:Gdiff} 
The function $G_\lambda:[0,\lambda]\to\left[0,\frac{\lambda}{\pi}\right]$ defined by  \eqref{eq:Glambda} is a strictly monotone decreasing convex $C^1$ function with
\[
\begin{alignedat}{3}
&&\qquad G_\lambda(0) &= \frac{\lambda}\pi,&\qquad G_\lambda(\lambda) &= 0,\\
G'_\lambda(z) &= - \frac{1}{\pi}\arccos \frac{z}{\lambda}, &\qquad G'_\lambda(0) &= - \frac{1}{2}, &\qquad G'_\lambda(\lambda) &= 0,\\
G''_\lambda(z) &= \frac{1}{\pi \sqrt{\lambda^2-z^2}}.&&&&
\end{alignedat}
\]
\end{lemma}
We can therefore define the inverse function $G^{-1}_\lambda:\left[0,\frac{\lambda}{\pi}\right]\to[0,\lambda]$ which is also monotone decreasing and convex.
Sometimes, it will be also convenient for us to consider $G_\lambda$ on the interval $\left[0, \ceiling{\lambda}\right]$ by extending it by zero to   $\left(\lambda, \ceiling{\lambda}\right]$: the resulting function, which we for simplicity denote by the same symbol, remains monotone decreasing, convex, and $C^1$.

\begin{lemma}\label{lem:Gint}
Let $\beta \ge 0$.
Then 
\[
\int_0^{\lambda} z^\beta G_\lambda(z)\, \dr z = 
\frac{\Gamma \left(\frac{\beta+1}2\right) \lambda^{\beta+2}}
{4 \sqrt\pi \,(\beta+2)\, \Gamma \left(\frac{\beta+4}2\right)}.
\]
In particular,
\begin{equation}\label{eq:Gint2}
\int_0^{\lambda} G_\lambda(z)\, \dr z = \frac{\lambda^2}{8}.
\end{equation}
\end{lemma}

\begin{proof}
In fact, the identity can be checked using computer algebra software, but we include a proof for the sake of completeness.
After a change of variables $z = \lambda \cos\tau$, we obtain
\begin{equation}
\label{eq:21}
\int_0^{\lambda} z^\beta G_\lambda(z)\, \dr z = \frac{\lambda^{\beta+2}}{\pi}
\int_0^{\pi/2} (\cos \tau)^\beta (\sin\tau - \tau\cos\tau) \sin\tau\,d\tau .
\end{equation}
By \cite[Eqs. 5.12.1--2]{dlmf},
\[
\int_0^{\pi/2} (\cos \tau)^\rho (\sin\tau)^\sigma \dr \tau 
= \frac{\Gamma\left(\frac{\rho+1}{2}\right) \Gamma\left(\frac{\sigma+1}{2}\right)}
{2\Gamma\left(\frac{\rho+\sigma+2}{2}\right)}
\]
for any $\rho, \sigma\ge 0$.
Therefore,
\[
\int_0^{\pi/2} (\cos \tau)^\beta (\sin\tau)^2 \dr\tau 
= \frac{\sqrt{\pi}\Gamma\left(\frac{\beta+1}{2}\right)}
{4\Gamma\left(\frac{\beta+4}{2}\right)},
\]
\[
\int_0^{\pi/2} \tau (\cos \tau)^{\beta+1} \sin\tau d\tau 
= - \left.\frac{\tau (\cos\tau)^{\beta+2}}{\beta+2}\right|_0^{\pi/2}
+ \frac1{\beta+2} \int_0^{\pi/2} (\cos \tau)^{\beta+2} d\tau 
= \frac{\sqrt\pi \,\Gamma \left(\frac{\beta+3}2\right)}
{2\,(\beta+2) \,\Gamma \left(\frac{\beta+4}2\right)},
\]
and so
\[
\int_0^{\pi/2} (\cos \tau)^\beta (\sin\tau - \tau\cos\tau) \sin\tau\,\dr\tau
= \frac{\sqrt{\pi}\Gamma\left(\frac{\beta+1}{2}\right)}
{4(\beta+2)\Gamma \left(\frac{\beta+4}{2}\right)}.
\]
Substituting this into \eqref{eq:21} we get the result.
\end{proof}

\begin{cor}\label{cor:GintW} Let $d\in\mathbb{N}$, $d\ge 2$. Then 
\[
\frac{2}{(d-2)!} \int_0^{\lambda} z^{d-2} G_\lambda(z)\, \dr z=W_d(\lambda)=w_d\lambda^d.
\]
\end{cor}

\begin{proof} Applying Lemma \ref{lem:Gint} with $\beta=d-2$ we get
\[
\frac{2}{(d-2)!} \int_0^{\lambda} z^{d-2} G_\lambda(z)\, \dr z= \frac{\Gamma \left(\frac{d-1}2\right) \lambda^d}
{2 \,d\,\sqrt\pi\, (d-2)! \, \Gamma \left(\frac{d+2}2\right)} .
\]
The duplication formula \cite[Eq. 5.5.5]{dlmf}
\[
\sqrt{\pi}(d-2)! = \sqrt{\pi}\Gamma(d-1) = 2^{d-2} \Gamma\left(\frac{d-1}{2}\right) \Gamma\left(\frac{d}{2}\right)
\]
implies
\[
\frac{\Gamma\left(\frac{d-1}{2}\right) \lambda^d}{2d\sqrt{\pi}(d-2)!\,\Gamma\left(\frac{d+2}2\right)}
= \frac{ \lambda^d}
{2^{d-1} \,d\,  \Gamma \left(\frac{d}2\right) \Gamma \left(\frac{d+2}2\right)}
= \frac{ \lambda^d}
{2^d \left(\Gamma \left(\frac{d+2}2\right)\right)^2}=W_d(\lambda).
\]
\end{proof}

An important role in our study in the Neumann case will be played by the inverse function value $G_\lambda^{-1}\left(\frac{1}{4}\right)$ (defined for all $\lambda\ge \frac{\pi}{4}$). We will use the following bounds.

\begin{lemma}\label{lem:Ginv} We have
\begin{equation}\label{eq:Ginv1}
G_\lambda^{-1}\left(\frac{1}{4}\right)< \lambda-1
\end{equation}
for all $\lambda\ge 2$. Additionally, for any $\sigma\in\left(0, \frac{\pi}{2}\right]$ we have
\begin{equation}\label{eq:Ginv2}
G_\lambda^{-1}\left(\frac{1}{4}\right)\ge \lambda\cos\sigma
\end{equation}
whenever
\begin{equation}\label{eq:Ginv3}
\lambda\ge r_1(\sigma):=\frac{\pi}{4(\sin\sigma-\sigma\cos\sigma)}.
\end{equation}
\end{lemma}

\begin{proof}
Since $G_\lambda$ is monotone decreasing, the claim \eqref{eq:Ginv1} is equivalent to
\begin{equation}\label{eq:Ginv4}
G_\lambda(\lambda-1)<\frac{1}{4}.
\end{equation}
We have $G_2(1)-\frac{1}{4}=\frac{\sqrt{3}}{\pi}-\frac{7}{12}<0$, and additionally
\[
\frac{\dr}{\dr\lambda}\left(G_\lambda(\lambda-1)\right) = \frac{1}{\pi}\left(\frac{1}{\lambda}\sqrt{2\lambda-1}-\arccos\left(1-\frac{1}{\lambda}\right)\right)<0
\]
as
\[
\cos\left(\frac{1}{\lambda}\sqrt{2\lambda-1}\right)>1-\frac{2\lambda-1}{2\lambda^2}>1-\frac{1}{\lambda},
\] 
thus implying \eqref{eq:Ginv4} for $\lambda\ge 2$. 

Similarly, the claim \eqref{eq:Ginv2} is equivalent to 
\[
G_\lambda(\lambda\cos\sigma)=\frac{\lambda(\sin\sigma-\sigma\cos\sigma)}{\pi}\ge \frac{1}{4},
\]
given \eqref{eq:Ginv3}. 
\end{proof}

\section{{Proof of Theorem \ref{thm:count1}}}\label{sec:proofD2}

We first state the following 
\begin{theorem}\label{thm:countDdisk}
Let $b>0$, and let $g$ be a non-negative decreasing convex function on $[0,b]$ such that  $g(b) = 0$ and
\begin{equation}\label{eq:gprime}
\left|g(z)-g(w)\right|\le \frac{1}{2}|z-w|
\end{equation}
for all $z,w\in[0,b]$.
Then
\begin{equation}\label{eq:gmaincount}
\entire{g(0)+\frac{1}{4}} + 2\sum_{m=1}^{\entire{b}} \entire{g(m)+\frac{1}{4}} \le
2\int_0^b g(z)\,\dr z.
\end{equation}
The equality is possible only if $g$ is identically zero on $[0,b]$.
\end{theorem}

\begin{remark}\label{rem:countDdisk}
We explain here, very informally,  the ideas behind the proof of Theorem \ref{thm:countDdisk}.
The area under the graph of the function $g$ on the interval
$[m,m+1]$ is approximately equal to the area under the straight
line passing through the points $(m, g(m))$ and $(m+1, g(m+1))$,
so 
\[
\int_m^{m+1} g(z) \dr z \approx \frac{1}{2} \left(g(m) + g(m+1)\right).
\]
Summing up these equalities over $m$ we obtain
\[
2 \int_0^b g(z) \dr z \approx g(0) + 2 \sum_{m=1}^{\entire{b}} g(m).
\]
If a number $x$ is chosen randomly then $\entire{x} \approx x -\frac{1}{2}$
on average. Thus, $\entire{g(m)+\frac{1}{4}}\approx g(m) -\frac{1}{4}$, and these
extra contributions of $ -\frac{1}{4}$ ensure the sign of the inequality in \eqref{eq:gmaincount}.
In order to prove Theorem \ref{thm:countDdisk} rigorously we divide the graph
by horizontal lines $y=n$, with $n = 0,1, \dots, \entire{g(0)}$, and
we consider what happens in the intervals where
$n+1 \ge g(z) \ge n$. The values of $\entire{g(m)+1/4}$ there are either 
$n$ or  $n+1$. The point $m$ is ``bad'' if $g(m) \ge n +\frac{3}{4}$ and thus  $\entire{g(m)+1/4}=n+1$: these ``bad'' points contribute more to the sum than we expect ``on average''.
The convexity of the function $g$ and condition \eqref{eq:gprime} ensure that the number
of such ``bad'' points is not greater than half of the total number of integer points 
in an interval, and this yields the required estimate
in the interval we are considering.
\end{remark}

In order to prove Theorem \ref{thm:countDdisk} we require the following 
\begin{lemma}\label{lem:Dcount}
Let $i, j \in \mathbb{Z}$, $i<j$.
Let $g$ be a decreasing convex function on $[i, j+1]$ satisfying \eqref{eq:gprime} for all $z,w\in[i,j+1]$. 
Assume additionally that
\begin{equation}\label{eq:gassumpt1}
n+1 \ge g(i+1) \ge \dots \ge g(j) \ge n \ge g(j+1)
\end{equation}
for some $n\in\ZZ$. Then
\begin{equation}\label{eq:32}
\frac{1}{2}\entire{g(i)+\frac{1}{4}}+\sum_{m=i+1}^{j-1} \entire{g(m)+\frac{1}{4}}+\frac{1}{2}\entire{g(j)+\frac{1}{4}} \le\int_{i}^{j} g(z)\,\dr z.
\end{equation}
\end{lemma}

\begin{proof}[Proof of Lemma \ref{lem:Dcount}]
The validity of the claim does not change if we add a constant integer number to the function $g$.
So, without loss of generality we can assume $n=0$, so that \eqref{eq:gassumpt1} becomes
\begin{equation}\label{eq:gassumpt2}
1 \ge g(i+1) \ge \dots \ge g(j) \ge 0 \ge g(j+1).
\end{equation}
Additionally, \eqref{eq:gprime} implies that $g(i)\le g(i+1)+\frac{1}{2}\le \frac{3}{2}$. 

Set 
\[
K=\#\left\{m\in\{i,\dots,j\}: \frac{3}{4} \le g(m)\right\},
\]
and consider four cases.

\begin{description}
\item[Case $K=0$.]  The left-hand side of  \eqref{eq:32} is zero, and the right-hand side is non-negative by \eqref{eq:gassumpt2}. 

\item[Case $K=1$.]  Here 
\[
\frac{5}{4}> g(i)\ge \frac{3}{4} > g(i+1)\ge\dots\ge g(j)\ge 0,
\] 
and the left-hand side of  \eqref{eq:32} is equal to $\frac{1}{2}$. The assumption \eqref{eq:gprime}  yields $g\left(i+\frac{1}{2}\right)\ge \frac{1}{2}$,
therefore by non-negativity and convexity of $g$ on $[i,j]$,
\[
\int_i^j g(z)\,\dr z  \ge \int_i^{i+1} g(z)\,\dr z\ge  g\left(i+\frac{1}{2}\right) \ge \frac{1}{2}.
\]

\item[Case $K=2$.]
Here 
\[
\frac{3}{2}\ge g(i)\ge g(i+1)\ge \frac{3}{4} > g(i+2)\ge \dots \ge g(j)\ge 0\ge g(j+1),
\] 
and the left-hand side of  \eqref{eq:32} equals $\frac{3}{2}$. By  \eqref{eq:gprime}  we have $j\ge i+2$, and therefore by non-negativity and convexity of $g$,
\[
\int_i^j g(z)\,\dr z \ge \int_i^{i+2} g(z)\,\dr z \ge 2 g(i+1) \ge \frac{3}{2}.
\]

\item[Case $K\ge 3$.]
Here
\[
\frac{3}{2}\ge g(i)\ge g(i+1)\ge \dots \ge g(i+K-1)\ge \frac{3}{4} > g(i+K)\ge \dots \ge g(j)\ge 0\ge g(j+1).
\] 
The left-hand side of  \eqref{eq:32} is equal to $K-\frac{1}{2}$. By convexity of $g$,
\[
(K-1) g(i+1) + (K-2) g(i+2K-2) \ge (2K-3) g(i+K-1),
\]
and therefore
\[
g(i+2K-2) \ge \frac{2K-5}{4(K-2)} > 0.
\]
Thus,  $j\ge i+2K-2$. Next,
\[
\int_i^j g(z)\,\dr z  \ge \int_i^{i+2K-2}g(z)\,\dr z  \ge (2K-2)g(i+K-1)\ge \frac{3K-3}{2}\ge K-\frac{1}{2}
\]
as $K\ge 3$. 
\end{description}
\end{proof}

\begin{remark}\label{rem:equality}
One can easily see from the proof that the equality in \eqref{eq:32} is attained in the following three cases only:
\begin{itemize}
\item $g(z) \equiv n$ on $[i,j]$ (if $K=0$);
\item $j=i+1$ and $g(z) = n + \frac{i-z}{2} + \frac{3}{4}$ on $[i,i+1]$ (if $K=1$);
\item $j=i+2$ and $g(z) = n + s (i+1-z) + \frac{3}{4}$ on $[i,i+2]$ with $s \in \left[\frac{3}{8}, \frac{1}{2}\right]$ (if $K=2$).
\end{itemize}
\end{remark} 

We can now proceed to the proof of Theorem \ref{thm:countDdisk} proper.

\begin{proof}[Proof of Theorem \ref{thm:countDdisk}]
Let 
\[
N = \entire{g(0)}.
\]
If $N=0$, then applying  Lemma \ref{lem:Dcount} with $i=0$, $j=\entire{b}$, and $g(z)$ extended by zero for $z\in\left(b,\entire{b}+1\right]$, gives
\[
\entire{g(0)+\frac{1}{4}}+2\sum_{m=1}^{\entire{b}} \entire{g(m)+\frac{1}{4}}\le 2\int_0^b g(z)\,\dr z,
\]
and therefore \eqref{eq:gmaincount}.

Assume now $N\ge 1$.
For $k = 0, 1, \dots, N$, let 
\begin{equation}\label{eq:Lk}
L_k := \max \left\{m \in \{0,\dots,\entire{b}\}: g(m) \ge k\right\}, 
\end{equation}
see Figure \ref{fig:4}.  Therefore, we have 
\[
0\le L_N< L_{N-1}< \dots < L_0=\entire{b},
\]
where the strict inequalities $L_k<L_{k-1}$, $k=1,\dots,N$, follow from \eqref{eq:gprime}.

\begin{figure}[htpb]
\centering
\includegraphics{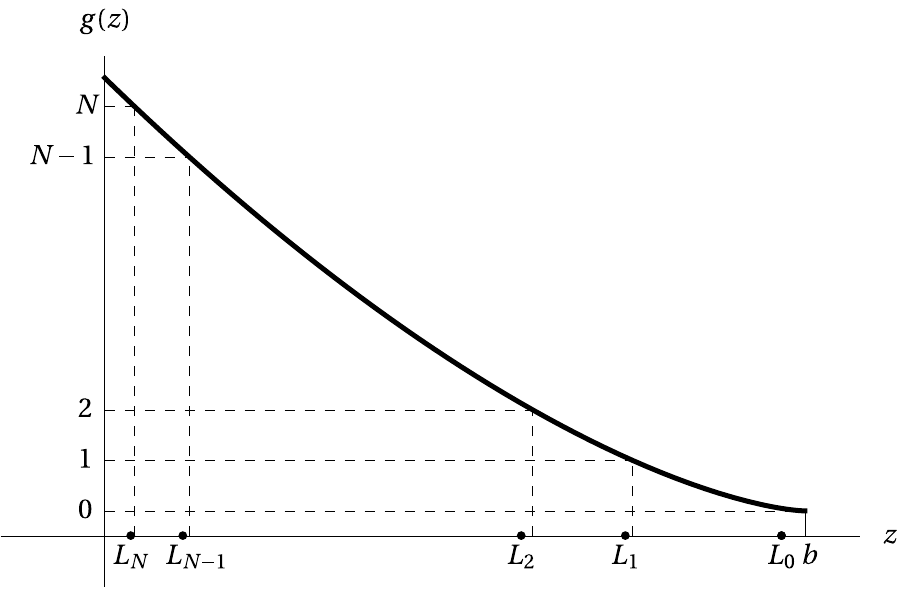}
\caption{The numbers $L_k$, see \eqref{eq:Lk} and also Remark \ref{rem:Lk}. \label{fig:4}}
\end{figure}

We will consider two cases depending on whether $L_N=0$. 
\begin{description}
\item[Case $L_N=0$.]  We write
\begin{equation}\label{eq:36}
\begin{split}
&\qquad\qquad\entire{g(0)+\frac{1}{4}}+2\sum_{m=1}^{\entire{b}} \entire{g(m)+\frac{1}{4}}\\
&= 2\sum_{k=0}^{N-1}\left(\frac{1}{2} \entire{g(L_{k+1})+\frac{1}{4}} + \sum_{m=L_{k+1}+1}^{L_k-1} \entire{g(m)+\frac{1}{4}}+\frac{1}{2} \entire{g(L_{k})+\frac{1}{4}}\right).
\end{split}
\end{equation}
For each $k\in\{0,\dots,N-1\}$ we have
\[
k+1 > g\left(L_{k+1}+1\right) \dots \ge g\left(L_k\right) \ge k\ge  g\left(L_k+1\right),
\]
and applying Lemma \ref{lem:Dcount} with $i=L_{k+1}$, $j=L_k$, and $n=k$ yields
\begin{equation}\label{eq:38}
\frac{1}{2} \entire{g(L_{k+1})+\frac{1}{4}} + \sum_{m=L_{k+1}+1}^{L_k-1} \entire{g(m)+\frac{1}{4}}+\frac{1}{2} \entire{g(L_{k})+\frac{1}{4}}\le \int_{L_{k+1}}^{L_k} g(z)\,\dr z.
\end{equation}
Substituting \eqref{eq:38} into \eqref{eq:36} gives
\begin{equation}\label{eq:381}
\entire{g(0)}+2\sum_{m=1}^{\entire{b}} \entire{g(m)+\frac{1}{4}} \le 2 \int_0^{\entire{b}} g(z)\,\dr z\le 2\int_0^b g(z)\,\dr z, 
\end{equation}
as required, where in the last inequality we used non-negativity of $g$.

\item[Case  $L_N>0$.]  We write
\begin{equation}\label{eq:39}
\begin{split}
&\qquad\qquad\entire{g(0)+\frac{1}{4}}+2\sum_{m=1}^{\entire{b}} \entire{g(m)+\frac{1}{4}}\\
&=2\left(\frac{1}{2}\entire{g(0)+\frac{1}{4}}+ \sum_{m=1}^{L_{N}-1} \entire{g(m)+\frac{1}{4}}+\frac{1}{2} \entire{g(L_N)+\frac{1}{4}}\right)\\
&+ 2\sum_{k=0}^{N-1}\left(\frac{1}{2} \entire{g(L_{k+1})+\frac{1}{4}} + \sum_{m=L_{k+1}+1}^{L_k-1} \entire{g(m)+\frac{1}{4}}+\frac{1}{2} \entire{g(L_{k})+\frac{1}{4}}\right).
\end{split}
\end{equation}
We have
\[
N+1>g(0) \ge \dots \ge g\left(L_{N}-1\right) \ge  g\left(L_{N}\right) \ge N > g\left(L_{N}+1\right),
\]
therefore applying Lemma \ref{lem:Dcount} with $i=0$, $j=L_N$, and $n=N$, we get
\begin{equation}\label{eq:310}
2\left(\frac{1}{2}\entire{g(0)+\frac{1}{4}}+ \sum_{m=1}^{L_{N}-1} \entire{g(m)+\frac{1}{4}}+\frac{1}{2} \entire{g(L_N)+\frac{1}{4}}\right)\le 2\int_0^{L_N} g(z)\,\dr z.
\end{equation}
Substituting \eqref{eq:310} and \eqref{eq:38} into \eqref{eq:39}  gives \eqref{eq:381}.
\end{description}

Finally, assume that we have the equality in \eqref{eq:gmaincount}.
Due to Remark \ref{rem:equality}, the function $g$ is linear on each interval $[L_{k+1}, L_k]$,
and either $\operatorname{dist}(g(L_k), \ZZ) \ge \frac{1}{4}$ or $g \equiv 0$ on $[L_{k+1}, L_k]$.
The situation when $\operatorname{dist}(g(L_k), \ZZ) \ge \frac{1}{4}$ and $g \equiv 0$ on $[L_k, L_{k-1}]$
is impossible due to continuity of $g$.
If $\operatorname{dist}(g(L_k), \ZZ) \ge \frac{1}{4}$ for all $k$, then in particular $g(L_0)=g(\entire{B})\ge \frac{1}{4}$, and  the last inequality in \eqref{eq:381} is strict.  
Therefore, the equality in \eqref{eq:gmaincount} requires $g \equiv 0$ on the whole interval $[0,b]$.
\end{proof}

\begin{remark}\label{rem:Lk} If $g$ is strictly monotone on $[0, b]$, then the inverse function $g^{-1}$ is well-defined on $[0, g(0)]$, and the definitions \eqref{eq:Lk} may be equivalently rewritten as
\[
L_k=\entire{g^{-1}(k)}, \qquad k=0,\dots,\entire{g(0)}.
\]   
\end{remark}

We finally use Theorem \ref{thm:countDdisk} to prove Theorem \ref{thm:count1}. 
\begin{proof}[Proof of Theorem \ref{thm:count1}]
We apply \eqref{eq:gmaincount}  with $b=\lambda$  and $g=G_\lambda$ (which we can do since Lemma \ref{lem:Gdiff} ensures that \eqref{eq:gprime} holds in this case), and use \eqref{eq:Gint2},  giving the bound \eqref{eq:ptD2} and therefore confirming the validity of the Dirichlet P\'olya's conjecture for the disk.
\end{proof}

\section{{Proof of Theorem \ref{thm:count3}}}\label{sec:proofN2}

We start by stating
\begin{theorem}\label{thm:countNdisk}
Let $b>0$, and let $g$ be a non-negative decreasing convex function on $[0,b]$ such that $g(0)\ge \frac{1}{4}$, $g(b)=0$, and \eqref{eq:gprime} holds for all $z,w\in[0,b]$. Let
\[
M_0=M_{g,0}:=1+\max\left\{m\in\{0,\dots,\entire{b}\}: g(m)\ge \frac{1}{4}\right\},
\]
and assume that $M_0\le b$. 
Then 
\begin{equation}\label{eq:60}
\sum_{m=0}^{\entire{b}} \entire{g(m)+\frac{3}{4}}\ge \int_0^b g(z)\,\dr z-\frac{b-3M_0}{8}.
\end{equation}
\end{theorem}

\begin{remark}\label{rem:M0} If $g$ is strictly monotone on $[0, b]$, then the inverse function $g^{-1}$ is well-defined on $[0, g(0)]$, and 
\begin{equation}\label{eq:M0inv}
M_0=\entire{g^{-1}\left(\frac{1}{4}\right)}+1,
\end{equation}  
cf.\ Remark \ref{rem:Lk}.
\end{remark}

\begin{remark}\label{rem:countNdisk}
Once more, we first outline a very informal plan of proving Theorem \ref{thm:countNdisk}.  
As we have argued in Remark \ref{rem:countDdisk}, we have
$\entire{g(m) + \frac{3}{4}}\approx g(m) + \frac{1}{4}$, which should in principle ensure the correct inequality sign in \eqref{eq:60}. ``Bad'' points
are now the points with $n \le g(m) \le n+  \frac{1}{4}$. So,
we divide the graph of $g$ by the horizontal lines at 
$y = n + \frac{1}{4}$, where $n = 0, 1, \dots, \entire{g(0)+\frac{3}{4}}$. Again,
this guarantees that the number of ``bad'' points in each resulting interval is less than half the total number of points there.
This still leaves an unresolved issue of points $m$ lying under the tail of the graph of $g$, where $0 \le g(m) < \frac{1}{4}$.
Such points make no contribution to the left-hand side of \eqref{eq:60}, but the tail does contribute to the integral: consider, for example, a toy case of a function $g(z)=\frac{10-z}{80}$ on the interval $[0,10]$. 
To account for that, we subtract
an additional term in the right-hand side of \eqref{eq:60}.
\end{remark}

Before proceeding to the proof of Theorem \ref{thm:countNdisk}, we require
\begin{lemma}\label{lem:62}
Let $i, j\in \mathbb{Z}$, $i<j$. Let $g$ be a decreasing convex function on $[i,j]$.
Assume that
\[
n+\frac{1}{4} > g(i) \ge \dots \ge g(j-1) \ge n-\frac{3}{4} > g(j)
\]
for some $n\in\mathbb{Z}$.
Then
\begin{equation}\label{eq:61}
\frac{1}{2}\entire{g(i) + \frac{3}{4}}
+ \sum_{m=i+1}^{j-1} \entire{g(m) + \frac{3}{4}}
+ \frac{1}{2}\entire{g(j) + \frac{3}{4}}
\ge \int_i^j g(z) \, \dr z + \frac{j-i}{4} - \frac{1}{2}.
\end{equation}
\end{lemma}

\begin{proof}[Proof of Lemma \ref{lem:62}] The left-hand side of \eqref{eq:61} is equal to $(j-i)n - \frac{1}{2}$.
By convexity of $g$,
\[
\int_i^j g(z) \, \dr z \le \frac{j-i}{2} \left(g(j) + g(i)\right) 
\le (j-i) \left(n - \frac{1}{4}\right).
\]
\end{proof}

\begin{proof}[Proof of Theorem \ref{thm:countNdisk}] As $g(0) \ge \frac{1}{4}$, we have
\begin{equation}\label{eq:NN}
N := \entire{g(0) + \frac{3}{4}} \ge 1.
\end{equation}
For $k = 1, \dots, N-1$, denote
\[
M_k = M_{g,k}:= 1+ \max \left\{m \in \left\{0,\dots,\entire{b}\right\}: g(m) \ge k + \frac14\right\}.
\]
We also denote $M_N = M_{g,N}:=0$.
The assumption \eqref{eq:gprime}  yields $M_k > M_{k+1}$ for all $k=0,\dots, N-1$.
Therefore, by Lemma \ref{lem:62},
\[
\begin{split}
&\qquad \frac12\entire{g(0)+\frac34} +
\sum_{m=1}^{\entire{b}} \entire{g(m)+\frac34} \\
&= \sum_{n=0}^{N-1} \left(\frac12 \entire{g\left(M_{n+1}\right)+\frac34} + \sum_{m=M_{n+1}+1}^{M_n-1} \entire{g(m)+\frac34} + \frac12 \entire{g\left(M_n\right)+\frac34}\right) \\
&\ge \sum_{n=0}^{N-1} \left(\int_{M_{n+1}}^{M_n} g(z)\, \dr z 
+ \frac{M_n - M_{n+1}}4 - \frac12\right)  
= \int_0^{M_0} g(z)\, \dr z  + \frac{M_0}4 - \frac{N}2.
\end{split}
\]
Due to \eqref{eq:NN} we get
\[
\sum_{m=0}^{\entire{b}} \entire{g(m)+\frac34} \ge \int_0^{M_0} g(z)\, \dr z  + \frac{M_0}4.
\]
By convexity of $g$,
\[
\int_{M_0}^b g(z)\, \dr z  \le
\frac{1}{2}g(M_0) (b-M_0) \le \frac{b-M_0}{8},
\]
and we get \eqref{eq:60}.
\end{proof}

We can now proceed to the proof of Theorem \ref{thm:count3}. We start with
\begin{prop}\label{prop:NlambdaNbound}
Let $\lambda\ge 2$. Then
\begin{equation}\label{eq:PtNR2} 
\Pt_2^\Neu(\lambda)> \frac{\lambda^2}{4}+\frac{1}{4}R_2(\lambda),
\end{equation} 
where 
\[
R_2(\lambda):=3G_\lambda^{-1}\left(\frac{1}{4}\right)-\lambda\left(1+\frac{4}{\pi}\right)-3.
\]
\end{prop}  

\begin{proof} 
We have, with account of $G_\lambda(0)=\frac{\lambda}{\pi}$,  
\[
\begin{split}
\Pt_2^\Neu(\lambda) &= \entire{G_\lambda(0)+\frac{3}{4}}+2\sum_{m=1}^{\entire{\lambda}}  \entire{G_\lambda(m)+\frac{3}{4}}\\
&=-\entire{\frac{\lambda}{\pi}+\frac{3}{4}}+2\sum_{m=0}^{\entire{\lambda}}  \entire{G_\lambda(m)+\frac{3}{4}}.
\end{split}
\]
We apply Theorem \ref{thm:countNdisk} to the sum in the right-hand side, with $g=G_\lambda$, $b=\lambda$, and 
\[
M_{G_\lambda,0}=\entire{G_\lambda^{-1}\left(\frac{1}{4}\right)}+1
\] 
(see Remark \ref{rem:M0}). Given that $\lambda\ge 2$, we take into account  the bound \eqref{eq:Ginv1} (which ensures that $M_{G_\lambda,0}\le\lambda$), and the value of the integral from Lemma \ref{lem:Gint}, leading to
\[
\Pt_2^\Neu(\lambda)- \frac{\lambda^2}{4}\ge 
\frac{3\entire{G_\lambda^{-1}\left(\frac{1}{4}\right)}+3-\lambda-4\entire{\frac{\lambda}{\pi}+\frac{3}{4}}}{4}.
\]
Finally, we use
\[
x-1<\entire{x}\le x
\]
to obtain \eqref{eq:PtNR2}.  
\end{proof}

We now have
\begin{prop}\label{prop:R2} 
Let $\lambda\ge \Lambda_1$, where $\Lambda_1$ is given by \eqref{eq:Lambda1}. Then $R_2(\lambda)\ge 0$.
\end{prop}

\begin{proof}[Proof of Proposition \ref{prop:R2}] 
Let $\sigma=\arccos\frac{5}{6}\in\left(0,\frac{\pi}{2}\right]$, and let
\[
\lambda\ge r_1(\sigma)=\frac{3\pi/2}{\sqrt{11}-5\arccos\frac{5}{6}}, 
\]
see \eqref{eq:Ginv3}. Then by Lemma \ref{lem:Ginv}, the bound \eqref{eq:Ginv2} holds.
Therefore,
\[
R_2(\lambda)\ge \lambda\left(3\cos\sigma -1-\frac{4}{\pi}\right)-3=\lambda\frac{3\pi-8}{2\pi}-3,
\]
and is non-negative if
\[
\lambda\ge\frac{6\pi}{3\pi-8}=\Lambda_1>0.
\]
As $\Lambda_1>r_1(\sigma)$
(with the inequality easy to prove rigorously using the method of verified rational approximations discussed in \S\ref{sec:Neumann2computer}), the result follows.
\end{proof}

Theorem \ref{thm:count3}  now immediately follows from Propositions \ref{prop:R2} and \ref{prop:NlambdaNbound}.

\section{{Proof of Theorem \ref{thm:count2}}}\label{sec:proofDd}.

We have the following ``dimension reduction'' formula.
\begin{theorem}\label{thm:dreduction}
Let $d\ge 3$. Then 
\begin{equation}\label{eq:dred}
\Pt^\Dir_d(\lambda)=\sum_{n=0}^{\entire{\lambda-\frac{d}{2}+1}}\binom{n+d-3}{d-3}\tilde{\Pt}^\Dir_{n+\frac{d}{2}-1}\left(\lambda\right),
\end{equation}   
where for $r\in[0,\lambda]$ we denote by
\begin{equation}\label{eq:PNDtilde}
\tilde{\Pt}^\Dir_{r}(\lambda):=\entire{\tilde{G}_{\lambda,r}(0)+\frac{1}{4}}+2\sum_{j=1}^{\entire{\lambda-r}}\entire{\tilde{G}_{\lambda,r}(j)+\frac{1}{4}}=
\entire{G_\lambda(r)+\frac{1}{4}}+2\sum_{j=1}^{\entire{\lambda-r}}\entire{G_{\lambda}(j+r)+\frac{1}{4}}
\end{equation}   
the ``standard'' two-dimensional weighted shifted lattice  point count under the graph of the function $\tilde{G}_{\lambda,r}(t):=G_\lambda(t+r)$, $t\in[0,\lambda-r]$. 
\end{theorem}

We remark that in comparison to our original definition 
\[
\Pt^\Dir_d(\lambda)=\sum_{m=0}^{\entire{\lambda-\frac{d}{2}+1}} \kappa_{d,m} \entire{G_\lambda\left(m+\frac{d}{2}-1\right)+\frac{1}{4}},
\]
see \eqref{eq:PtDN}, where the weights $\kappa_{d,m}$ are attached at each individual abscissa $m$, the formula in the right-hand side of \eqref{eq:dred} attaches weights $\binom{n+d-3}{d-3}$ to the whole counts $\tilde{\Pt}^\Dir_{n+\frac{d}{2}-1}\left(\lambda\right)$, which we will later estimate using the previously proven Theorem \ref{thm:countDdisk}. Note also that for $\lambda<\frac{d}{2}-1$, the equation \eqref{eq:dred} becomes the trivial identity $0=0$ by our notational convention for sums, see Remark \ref{rem:NDballfinite}.

\begin{proof}[Proof of Theorem \ref{thm:dreduction}]
With account of \eqref{eq:PNDtilde}, the right-hand side of \eqref{eq:dred} reads
\[
\sum_{n=0}^{\entire{\lambda-\frac{d}{2}+1}}\binom{n+d-3}{d-3}\left(\entire{G_\lambda\left(n+\frac{d}{2}-1\right)+\frac{1}{4}}+2\sum_{j=1}^{\entire{\lambda-n-\frac{d}{2}+1}}\entire{G_{\lambda}\left(j+n+\frac{d}{2}-1\right)+\frac{1}{4}}\right).
\]
For a fixed $m\in\{0\}\cup \mathbb{N}$, the contribution of $\entire{G_\lambda\left(m+\frac{d}{2}-1\right)+\frac{1}{4}}$ in this expression appears with the factor
\[
\begin{split}
\binom{m+d-3}{d-3}+2\sum_{j=0}^{m-1} \binom{j+d-3}{d-3}&=\sum_{j=0}^{m} \binom{j+d-3}{d-3}+\sum_{j=0}^{m-1} \binom{j+d-3}{d-3}\\
&=\sum_{i=d-3}^{m+d-3} \binom{i}{d-3}+\sum_{i=d-3}^{m-d-4} \binom{i}{d-3}= \binom{m+d-2}{d-2}+ \binom{m+d-3}{d-2}\\
&=\binom{m+d-1}{d-1}-\binom{m+d-3}{d-1}=\kappa_{d,m},
\end{split}
\]
where we have used the standard identity \cite[Eq. 26.3.7]{dlmf} 
\begin{equation}\label{eq:binom1}
\sum_{i=l}^{r}\binom{i}{l}=\binom{r+1}{l+1},\qquad r\ge l,
\end{equation}
and another identity \cite[Eq. 26.3.5]{dlmf},
\[
\binom{k}{l}=\binom{k+1}{l+1}-\binom{k}{l+1},\qquad k\ge l.
\]
Thus, the contributions of $\entire{G_\lambda\left(m+\frac{d}{2}-1\right)+\frac{1}{4}}$ in both sides of \eqref{eq:dred} coincide.
\end{proof}

Before proceeding to the proof of Theorem \ref{thm:count2} we will introduce some additional notation and state some auxiliary facts which will be required later.
Let, for  $x\ge 0$, 
\[
\Pi_n(x):=\prod_{j=1}^n (x+j)=(x+1)\cdots(x+n)\quad\text{for }n\in\mathbb{N},\qquad \Pi_0(x):=1.
\]
The function $\Pi_n(x)$ is closely related to \emph{Pochhammer's symbol}, or the \emph{rising factorial} $(x)_n:=x\cdots (x+n-1)$  (for which numerous other notation is also used) in the sense that $\Pi_n(x)=(x+1)_n$. We also have
\[
\binom{i+j}{i}=\frac{(i+1)\cdots (i+j)}{j!}=\frac{\Pi_j(i)}{j!}=\frac{\Pi_i(j)}{i!}.
\] 

Let us introduce, for $d\ge 3$, a piecewise-constant function
\begin{equation}\label{eq:fd0}
f_d(t):=\begin{cases}
0,&\qquad\text{if }t<\frac{d}{2}-1,\\
\frac{\Pi_{d-2}(m)}{(d-2)!}=\binom{m+d-2}{d-2},&\qquad\text{if }\entire{t-\frac{d}{2}+1}=m\ge 0,
\end{cases}
\end{equation}
and let
\[
F_d(z):=\int_0^z f_d(t)\,\dr t.
\]

In what follows we will require an upper polynomial bound on $F_d(z)$. If $z<\frac{d}{2}-1$, then $F_d(z)=0$. Let $z\ge \frac{d}{2}-1$ and $m= \entire{z-\frac{d}{2}+1}\ge 0$. 
Then
\begin{equation}\label{eq:Fd1}
\begin{split}
F_d(z)=\int_0^z f_d(t)\,\dr t&=\left(\sum_{k=0}^{m-1}\int_{k+d/2-1}^{k+d/2} f_d(t)\,\dr t\right)+\int_{m+d/2-1}^{z} f_d(t)\,\dr t\\
&=\left(\sum_{k=0}^{m-1}\binom{k+d-2}{d-2}\right)+\left(z-m-\frac{d}{2}+1\right)\binom{m+d-2}{d-2}\\
&=\binom{m+d-2}{d-1}+\left(z-m-\frac{d}{2}+1\right)\binom{m+d-2}{d-2}\\
&=\frac{\Pi_{d-2}(m)}{(d-1)!}\left((d-1)z-(d-2)m-\frac{(d-1)(d-2)}{2}\right),
\end{split}
\end{equation}
where we used \eqref{eq:binom1} to evaluate the sum of binomial coefficients. 

To establish a bound on $F_d(z)$, we apply the AM-GM inequality
\[
l\left(\beta_1\cdots \beta_{\ell}\right)^{1/l}\le \beta_1+\dots+\beta_l,\qquad l\in\mathbb{N},\quad \beta_1,\dots,\beta_l\ge 0, 
\]
with $l=d-1$ and 
\[
\beta_j=\begin{cases}
(m+j)\Pi_{d-1}(m)&\qquad\text{if }j\le d-2,\\
z^{d-1}&\qquad\text{if }j=d-1,
\end{cases}
\]
yielding 
\[
\begin{split}
(d-1)z\Pi_{d-2}(m)&\le \Pi_{d-2}(m)\left((m+1)+\dots+(m+d-2)\right)+z^{d-1}\\
&=\left((d-2)m+\frac{(d-2)(d-1)}{2}\right)\Pi_{d-2}(m)+z^{d-1}.
\end{split}
\]
Collecting together the terms with $\Pi_{d-2}(m)$ and substituting the resulting bound into the right-hand side of \eqref{eq:Fd1}, we deduce that 
\begin{equation}\label{eq:Fbounds}
F_d(z)\le \tilde{F}_d(z):=\frac{z^{d-1}}{(d-1)!}\qquad\text{for all }z\ge 0.
\end{equation}

We will require an auxiliary
\begin{lemma}\label{lem:prodint} 
Let $f$ be a locally integrable function on $[0,+\infty)$, and let $\tilde{F}\in C^1[0,+\infty)$ with $\tilde{F}(0)=0$ be such that
\[
F(z):=\int_0^z f(t)\,\dr t\le \tilde{F}(z)\qquad\text{for all }z\ge 0.
\] 
Let $b>0$, and let $g\in C^1[0,b]$ be a decreasing function such that $g(b)=0$. Then
\begin{equation}\label{eq:fgprod}
\int_0^b f(z)g(z)\,\dr z\le  \int_{0}^b \tilde{F}'(z) g(z)\,\dr z. 
\end{equation}
\end{lemma}

\begin{proof}
Integrating by parts, we get
\[
\begin{split}
 \int_{0}^b \tilde{F}'(z) g(z)\,\dr z-  \int_0^b f(z)g(z)\,\dr z&= \int_0^b \left(\tilde{F}'(z)-F'(z)\right) g(z) \,\dr z\\
&= \left.\left(\tilde{F}(z)-F(z)\right)g(z)\right|_0^b-\int_0^b \left(\tilde{F}(z)-F(z)\right)g'(z) \,\dr z\ge 0,
\end{split}
\]
since $\tilde{F}(0)=F(0)=g(b)=0$ and $\left(\tilde{F}(z)-F(z)\right)g'(z)\le 0$, thus proving \eqref{eq:fgprod}.
\end{proof}

We now proceed to the proof of Theorem \ref{thm:count2} proper. 

\begin{proof}[Proof of Theorem \ref{thm:count2}]  First of all, note that for $\lambda<\frac{d}{2}-1$ there is nothing to prove as in this case $\Pt^\Dir_d(\lambda)=0$.
For $\lambda\ge \frac{d}{2}-1$, we apply Theorem \ref{thm:countDdisk} with $g=\tilde{G}_{\lambda,r}$, $r=n+\frac{d}{2}-1$, and $b=\lambda-r=\lambda-n-\frac{d}{2}+1$ to the right-hand side of \eqref{eq:PNDtilde}, giving
\[
\tilde{\Pt}^\Dir_{n+\frac{d}{2}-1} (\lambda)< 2\int_0^{\lambda-n-\frac{d}{2}+1} G_\lambda\left(t+n+\frac{d}{2}-1\right)\,\dr t = 2\int_{n+\frac{d}{2}-1}^\lambda G_\lambda(z)\,\dr z
\]
for $n=0,\dots,\entire{\lambda-\frac{d}{2}+1}$. We now substitute the results into \eqref{eq:dred}, yielding, with account of \eqref{eq:binom1}, the bound
\begin{equation}\label{eq:PDdbound1}
\begin{split}
\Pt^\Dir_d(\lambda)&<2\sum_{n=0}^{\entire{\lambda-\frac{d}{2}+1}}\binom{n+d-3}{d-3}\int_{n+\frac{d}{2}-1}^\lambda G_\lambda(z)\,\dr z\\
&=2\int_{\frac{d}{2}-1}^\lambda \left(\sum_{n=0}^{\entire{z-\frac{d}{2}+1}}\binom{n+d-3}{d-3}\right)G_\lambda(z)\,\dr z\\
&=2\int_{\frac{d}{2}-1}^\lambda \binom{\entire{z-\frac{d}{2}+1}+d-2}{d-2} G_\lambda(z)\,\dr z.
\end{split}
\end{equation}
Using notation \eqref{eq:fd0}, we rewrite \eqref{eq:PDdbound1} as
\begin{equation}\label{eq:PDdbound2}
\Pt^\Dir_d(\lambda)<2\int_0^\lambda f_d(z)G_\lambda(z)\,\dr z.
\end{equation}
We now apply Lemma \ref{lem:prodint}  to the right-hand side of \eqref{eq:PDdbound2}, taking $f=f_d$, $g=G_\lambda$, and $\tilde{F}(z)=\tilde{F}_d(z)=\frac{z^{d-1}}{(d-1)!}$ by \eqref{eq:Fbounds},  which gives
\[
\Pt^\Dir_d(\lambda) < 2\int_0^\lambda \tilde{F}'_d(z)G_\lambda(z)\,\dr z=\frac{2}{(d-2)!}\int_0^\lambda z^{d-2}G_\lambda(z)\,\dr z.
\]
Finally, applying Corollary \ref{cor:GintW} gives
\[
\Pt^\Dir_d(\lambda) < W_d(\lambda)
\]
as required.
\end{proof}
 
\section{{Closing the gap: the proof of Theorem \ref{thm:count4}}}\label{sec:Neumann2computer}

We describe the algorithm (based on the two Principles stated in \S\ref{sec:WL}) of verifying the statement
\begin{equation}\label{eq:Nboundneeded}
 \Pt_2^\Neu(\lambda)>\frac{\lambda^2}{4}\qquad\text{for all }\lambda\in[3,14],
\end{equation}
with the lattice point counts given explicitly by
\[
\Pt^\Neu_2(\lambda)=\sum_{m=0}^{\entire{\lambda}}\kappa_{2,m}\entire{G_\lambda\left(m\right)+\frac{3}{4}}.
\]
 
The first Principle is realised with the help of the following simple lemma, which ensures that our algorithm described below requires only a finite number of steps.
 
\begin{lemma} 
If the inequality \eqref{eq:Nboundneeded} holds for a particular $\lambda_0$, that is, we have 
\[
e(\lambda_0):=\Pt_2^\Neu(\lambda_0)-\frac{\lambda_0^2}{4}>0, 
\]
this inequality also holds  for all 
$\lambda\in\left(\lambda_0,\lambda_0+\delta(\lambda_0)\right)=\left(\lambda_0,\sqrt{\lambda_0^2+4e(\lambda_0)}\right)$,
where
\[
\delta(\lambda_0):=\sqrt{\lambda_0^2+4e(\lambda_0)}-\lambda_0.
\]
\end{lemma} 

\begin{proof}
The result immediately follows from the facts that $\Pt^\Neu_2(\lambda)$ is  non-decreasing in $\lambda$ and that $\Lambda=\lambda_0+\delta(\lambda_0)$ is the positive root of the equation
$\frac{\Lambda^2}{4}=\frac{\lambda_0^2}{4}+e(\lambda_0)$.
\end{proof}

To implement the second Principle, we work with rational numbers only. Let, for $x\in\mathbb{R}$, $\underline{x}\le x\le \overline{x}$, where $\underline{x}, \overline{x}\in\Q$ are some lower and upper rational approximations of $x$. The function $G_\lambda(z)$ does not as a rule take rational values even for rational $\lambda$ and $z$, so to overcome this we work instead with 
\[
\underline{\Pt}^\Neu_2(\lambda)=\sum_{m=0}^{\entire{\lambda}}\kappa_{2,m}\entire{\underline{G}_\lambda\left(m\right)+\frac{3}{4}}, \qquad\lambda\in\Q,
\]
where
\[
\underline{G}_\lambda(z)=\frac{1}{\overline{\pi}}\left(\underline{\sqrt{\lambda^2-z^2}}-z\arccosU\frac{z}{\lambda}\right),\qquad z\in\mathbb{Q}\cap[0,\lambda].
\]

Of course, $\underline{G}_\lambda(z)\le G_\lambda(z)$ and therefore $\underline{\Pt}^\Neu_2(\lambda)\le \Pt^\Neu_2(\lambda)$, for $\lambda,z\in\Q$.
Obviously, taking integer parts of rational numbers, as well as other arithmetic operations on them is exact and does not introduce any numerical errors. We now describe how to construct \emph{verified rational approximations} of the square roots $\underline{\sqrt{\cdot}}$ and arccosines $\arccosU(\cdot)$. For the former, any guess (say, obtained from numerics) can be directly verified by taking squares and comparing rationals, which is rigorous. To verify our approximations of arccosines (taken at rational points) one may proceed as follows. Define the functions $\cosL, \cosU: \Q\cap\left(0,\frac{\pi}{2}\right]\to\Q$ as
\[
\cosU x := T_{12}[\cos](x), \qquad 
\cosL x := T_{14}[\cos](x),
\]
where $T_{t}[\cos](x)$ is the Taylor polynomial of $\cos x$ at $x=0$ of degree $t$, so that $\cosL x<\cos x<\cosU x$, see, for example, \cite[Problem 15]{Dor}.\footnote{We thank the referee for pointing out this reference.} 
Then 
\[
\cosU\left(\overline{\beta}\right)<x<\cosL\left(\underline{\beta}\right)\text{\qquad implies\qquad}\underline{\beta}:=\arccosL x<\beta=\arccos x<\arccosU x=:\overline{\beta}, 
\]
and the verification is again reduced to elementary operations on rationals. In the same manner,
\[
\underline{\pi}=3\arccosL\frac{1}{2}\quad\text{and}\quad \overline{\pi}=3\arccosU\frac{1}{2}
\]
provide verified rational approximations for $\pi$.

To finish describing our process, we need also to rationalise the square root appearing in the definition of $\delta(\lambda)$: 
we effectively  replace $e(\lambda)$ by a smaller number 
\begin{equation}\label{eq:eL}
\underline{e}(\lambda):=\underline{\Pt}^\Neu_2(\lambda)-\frac{\lambda^2}{4},
\end{equation} 
and also replace 
$\delta(\lambda)$ 
by a smaller number 
\begin{equation}\label{eq:deltaL}
\underline{\delta}(\lambda):= \underline{\sqrt{\lambda^2+4\underline{e}(\lambda)}}-\lambda,
\end{equation}  
where a verification is again by taking squares.

\begin{remark}\label{rem:fractions} In practice, we use the following process to find lower and upper rational approximations of a number $x\in\mathbb{R}$ (which may be a square root, or an arccosine). Throughout, we fix a relatively small number $\epsilon$ (say, $\epsilon=10^{-3}$) as an accuracy parameter. We find numerically some approximation $x_0$ of $x$  (which may be above or below $x$) with some better accuracy.  Then, we define $\underline{x}$ as the rational number in the interval $[x_0-3\epsilon, x_0-\epsilon]$  with the smallest possible denominator, and $\overline{x}$ as the rational number in the interval $[x_0+\epsilon, x_0+3\epsilon]$ with the smallest possible denominator, using a modification of a fast algorithm for traversing the Stern--Brocot tree \cite{Fori}. As we always verify the resulting approximations using the procedures described above, we do not in fact depend on the quality of an original numerical ``guess'' $x_0$ as long as $|x_0-x|<\epsilon$. 
\end{remark}

Thus, our main algorithms work as follows. In order to prove \eqref{eq:Nboundneeded} for $\lambda\in[\Lambda_0, \Lambda_1]$, we move upwards: set $\lambda=\underline{\Lambda_{0}}$, compute the margin $\underline{e}(\lambda)$ from \eqref{eq:eL}, set $\lambda_\text{new}=\lambda+\underline{\delta}\left(\lambda\right)$ using \eqref{eq:deltaL}, and continue the process.  If the margins are positive on each step, the process will stop successfully if after a finite number of steps we reach  $\lambda> \overline{\Lambda_{1}}$, see Figure \ref{fig:algorithms}.
  
\begin{figure}[htb]
\begin{center}\begin{minipage}{0.5\textwidth}
\begin{algorithm}[H]
\begin{algorithmic}
\State $\lambda\gets\underline{\Lambda_{0}}$
\State $\text{stepnumber}\gets 0$
\While{$\lambda\le \overline{\Lambda_{1}}$} 
\State $\text{stepnumber}\gets \text{stepnumber}+1$
\State $p\gets  \underline{\Pt}^\Neu_2(\lambda)$
\State $e\gets p-\frac{\lambda^2}{4}$
\If{$e>0$}
\State 
\State $\lambda\gets \underline{\sqrt{\lambda^2+4e}}$
\Else
\State \textbf{print} ``Proof failed  \Frowny{}''; \textbf{exit}
\EndIf
\EndWhile
\State \textbf{print} ``Success in ",  \text{stepnumber}, ``steps \Smiley{}''
\end{algorithmic}
\caption*{\textbf{The algorithm for proving \eqref{eq:Nboundneeded}}}
\end{algorithm}
\end{minipage}
\end{center}
 
\

\caption{The basic algorithm.}\label{fig:algorithms}
\end{figure}

The algorithm works extremely fast (when implemented in \texttt{Mathematica}, see the footnote on the title page), thus proving Theorem \ref{thm:count4}: in principle, with enough patience the whole implementation can be done by hand.
We summarise its outcomes in  Table \ref{table:2}. 
\begin{table}[htb]
{\setlength{\tabcolsep}{10pt}
\centering
\begin{longtable}{@{}>{$}c<{$}*{3}{>{$}c<{$}}@{}}
\toprule
\text{Step}&\lambda&\underline{e}(\lambda)&\underline{\delta}(\lambda)\\\addlinespace[1.5ex]
\midrule
 1 & 3<\Lambda_0 & \frac{3}{4} & \frac{6}{13} \\\addlinespace[1.5ex]
 2 & \frac{45}{13} & \frac{1355}{676} & \frac{223}{221} \\\addlinespace[1.5ex]
 3 & \frac{76}{17} & \frac{868}{289} & \frac{584}{493} \\\addlinespace[1.5ex]
 4 & \frac{164}{29} & \frac{3368}{841} & \frac{995}{783} \\\addlinespace[1.5ex]
 5 & \frac{187}{27} & \frac{11687}{2916} & \frac{29}{27} \\\addlinespace[1.5ex]
 6 & 8 & 3 & \frac{43}{60} \\\addlinespace[1.5ex]
 7 & \frac{523}{60} & \frac{57671}{14400} & \frac{227}{260} \\\addlinespace[1.5ex]
 8 & \frac{374}{39} & \frac{6098}{1521} & \frac{719}{897} \\\addlinespace[1.5ex]
 9 & \frac{239}{23} & \frac{10591}{2116} & \frac{339}{368} \\\addlinespace[1.5ex]
 10 & \frac{181}{16} & \frac{4103}{1024} & \frac{11}{16} \\\addlinespace[1.5ex]
 11 & 12 & 6 & \frac{24}{25} \\\addlinespace[1.5ex]
 12 & \frac{324}{25} & \frac{2506}{625} & \frac{241}{400} \\\addlinespace[1.5ex]
 13 & \frac{217}{16} & \frac{7183}{1024} & \frac{271}{272} \\\addlinespace[1.5ex]
 14 & \frac{495}{34}>\Lambda_1 & & \\\addlinespace[1.5ex]
  \bottomrule\\\addlinespace[2ex]
   \caption{Detailed output of the computer-assisted algorithm.}\label{table:2}
\end{longtable}}
\end{table}

\section{{Proof of Theorem \ref{thm:polyasector}}}\label{sec:sectors}
Let $S_\alpha$ be a circular sector of aperture $0<\alpha\le 2\pi$. The eigenvalues of the Dirichlet and Neumann Laplacians on $S_\alpha$ are easily found by separation of variables. They are all simple, and are given by 
\[
\lambda_{m,k}=\left(j_{\frac{m\pi}{\alpha},k}\right)^2\qquad m\in\mathbb{N},\quad k\in\mathbb{N},
\]
and
\[
\mu_{m,k}=\left(j'_{\frac{m\pi}{\alpha},k}\right)^2,\qquad m\in\mathbb{N}\cup\{0\},\quad k\in\mathbb{N},
\]
respectively. Therefore, the corresponding eigenvalue counting functions are
\[
\N^\Dir_{S_\alpha}(\lambda)=\sum_{m=1}^{\entire{\frac{\alpha\lambda}{\pi}}}\#\left\{k\in\mathbb{N}:j_{\frac{m\pi}{\alpha},k}\le\lambda\right\}\qquad\text{and}\qquad
\N^\Neu_{S_\alpha}(\lambda)=\sum_{m=0}^{\entire{\frac{\alpha\lambda}{\pi}}}\#\left\{k\in\mathbb{N}:j'_{\frac{m\pi}{\alpha},k}\le\lambda\right\},
\]
where the both sums are finite since $j_{\nu,1}>j'_{\nu,1}\ge \nu$.

Assume for the moment that the sector $S_\alpha$ contains a half-disk, that is $\alpha\in[\pi,2\pi]$. By Proposition \ref{prop:mainbound} with account of \eqref{eq:Ams}, \eqref{eq:Glambda} and \eqref{eq:saleph}, we have
\begin{equation}\label{eq:NDsectineq}
\N^\Dir_{S_\alpha}(\lambda) \le \sum_{m=1}^{\entire{\frac{\alpha\lambda}{\pi}}}\entire{G_\lambda\left(\frac{m\pi}{\alpha}\right)+\frac{1}{4}}.
\end{equation}
Set
\[
g_{\lambda,\alpha}(t):=G_\lambda\left(\frac{\pi t}{\alpha}\right), \qquad t\in\left[0, b\right], 
\]
where 
\[
b:=\frac{\alpha\lambda}{\pi}.
\]
Then $g_{\lambda,\alpha}$ is a monotone decreasing convex function on $[0,b]$ with $g(b)=0$; moreover,
\[
g'_{\lambda,\alpha}(t)=\frac{\pi}{\alpha}G'_\lambda\left(\frac{\pi t}{\alpha}\right)\in \left[-\frac{\pi}{2\alpha},0\right]\subset\left[-\frac{1}{2},0\right]
\]
due to our assumption $\alpha\ge \pi$. With this notation, the right-hand side of \eqref{eq:NDsectineq} becomes
\[
\sum_{m=1}^{\entire{b}}\entire{g_{\lambda,\alpha}(m)+\frac{1}{4}},
\]
and we can estimate it from above directly by Theorem \ref{thm:countDdisk} with $g=g_{\lambda,\alpha}$, giving
\[
\sum_{m=1}^{\entire{b}}\entire{g_{\lambda,\alpha}(m)+\frac{1}{4}}<\int_0^{\alpha\lambda/\pi}g_{\lambda,\alpha}(t)\,\dr t=\frac{\alpha}{\pi}\int_0^\lambda G_\lambda(z)\,\dr z=\frac{\alpha \lambda^2}{8\pi}. 
\]
Substituting this into \eqref{eq:NDsectineq} proves the Dirichlet P\'olya's conjecture for sectors containing a half-disk.

We now turn to the Neumann problem in $S_\alpha$, still assuming that $\alpha\in[\pi,2\pi]$. Following the same argument as in Lemma \ref{lem:Lambda0}, we conclude that the Neumann P\'{o}lya's conjecture for $S_\alpha$ holds for $\lambda\le 2\sqrt{\frac{6\pi}{\alpha}}$. We therefore may assume that 
\begin{equation}\label{eq:lambdaNsecmin}
\lambda>2\sqrt{\frac{6\pi}{\alpha}}\ge 2\sqrt{3}
\end{equation}
from now on. 

Using once more Proposition \ref{prop:mainbound} yields
\begin{equation}\label{eq:NNsectineq}
\N^\Neu_{S_\alpha}(\lambda) \ge \sum_{m=0}^{\entire{\frac{\alpha\lambda}{\pi}}}\entire{G_\lambda\left(\frac{m\pi}{\alpha}\right)+\frac{3}{4}}=
\sum_{m=0}^{\entire{b}}\entire{g_{\lambda,\alpha}(m)+\frac{3}{4}},
\end{equation}
with the same function $g_{\lambda,\alpha}$ and parameter $b$ as above.
 
We are now going to use Theorem \ref{thm:countNdisk} with $g=g_{\lambda,\alpha}$ to estimate the right-hand side of \eqref{eq:NNsectineq} from below. 
Note that $g_{\lambda,\alpha}(0)=\frac{\lambda}{\pi}>\frac{1}{4}$ by \eqref{eq:lambdaNsecmin}. Also, by \eqref{eq:M0inv}
\[
M_0=M_{g_{\lambda,\alpha},0}=\entire{g_{\lambda,\alpha}^{-1}\left(\frac{1}{4}\right)}+1=\entire{\frac{\alpha}{\pi}G_\lambda^{-1}\left(\frac{1}{4}\right)}+1,
\] 
and in order to apply Theorem \ref{thm:countNdisk}  we need to ensure that $M_0\le b=\frac{\alpha\lambda}{\pi}$. But this is true since, with account of  $\alpha\ge \pi$, 
\[
M_0-b\le \frac{\alpha}{\pi}G_\lambda^{-1}\left(\frac{1}{4}\right)+1-\frac{\alpha\lambda}{\pi}\le \frac{\alpha}{\pi}\left(G_\lambda^{-1}\left(\frac{1}{4}\right)+1-\lambda\right)<0
\]
by \eqref{eq:Ginv1}. Then, \eqref{eq:60} implies
\begin{equation}\label{eq:bMsector}
\sum_{m=0}^{\entire{b}}\entire{g_{\lambda,\alpha}(m)+\frac{3}{4}}\ge \frac{\alpha \lambda^2}{8\pi}-\frac{b-3M_0}{8}.
\end{equation}
We now reason as in the proof of Theorem \ref{thm:count3}: if we can show that $b-3M_0<0$ for all $\lambda\ge 2\sqrt{3}$, this would prove, via the combination of \eqref{eq:NNsectineq} and \eqref{eq:bMsector}, that $\N^\Neu_{S_\alpha}(\lambda)>\frac{\alpha \lambda^2}{8\pi}$. We have
\[
b-3M_0=\frac{\alpha\lambda}{\pi}-3\entire{\frac{\alpha}{\pi}G_\lambda^{-1}\left(\frac{1}{4}\right)}-3
\le \frac{\alpha}{\pi}\left(\lambda-3G_\lambda^{-1}\left(\frac{1}{4}\right)\right)
< \frac{\alpha}{\pi}\left(\lambda-2G_\lambda^{-1}\left(\frac{1}{4}\right)\right).
\]
We now apply the second statement of Lemma \ref{lem:Ginv} with $\sigma=\frac{\pi}{3}$ which guarantees that $b-3M_0<0$ for 
\[
\lambda>r_1\left(\frac{\pi}{3}\right)=\frac{3\pi}{6\sqrt{3}-2\pi}.
\]
As $\frac{3\pi}{6\sqrt{3}-2\pi}<2\sqrt{3}$, this finishes the proof of  Theorem \ref{thm:polyasector} for sectors $S_\alpha$ of aperture $\alpha\in[\pi, 2\pi]$. 

To complete the proof of  Theorem \ref{thm:polyasector} we now need to consider the case $\alpha\in(0,\pi)$. Set
\[
\ell:=\entire{\frac{2\pi}{\alpha}}\ge 2.
\]
Then $\tilde{\alpha}=\ell\alpha\in[\pi,2\pi]$, P\'olya's conjecture holds for $S_{\tilde{\alpha}}$, and $S_\alpha$ tiles  $S_{\tilde{\alpha}}$. By Theorem \ref{thm:tiling}, P\'olya's conjecture holds for $S_\alpha$.

\end{document}